\documentclass[12pt, reqno]{amsart}
\usepackage{amssymb, hyperref}
\usepackage{amsthm}
\usepackage{graphicx}
\usepackage[all,2cell]{xy}
\usepackage{verbatim}
\usepackage{tikz}
\usepackage{tikz-cd}
\usepackage{hyperref}
\usetikzlibrary{matrix,arrows,decorations.pathmorphing}
\usepackage{mathrsfs}
\numberwithin{equation}{section}
\newtheorem*{theorem*}{Theorem}
\newtheorem*{corollary*}{\bf Corollary}
\usepackage{lipsum}
\usetikzlibrary{matrix,arrows,decorations.pathmorphing}
\usepackage{lineno}
\newtheorem{theorem}{Theorem}[section]
\newtheorem{corollary}[theorem]{Corollary}
\newtheorem{definition}[theorem]{Definition}
\newtheorem{example}[theorem]{Example}
\newtheorem{lemma}[theorem]{Lemma}

\newtheorem{proposition}[theorem]{Proposition}
\newtheorem{remark}[theorem]{Remark}
\usepackage{csquotes}
\usepackage{epigraph}

\title[On Mori cone of  Bott towers]{On Mori cone of Bott towers}
 \author{B. Narasimha Chary}
\address{%
B. Narasimha Chary\\
Institut Fourier, UMR 5582 du CNRS\\
Universit{\'e} de Grenoble Alpes\\
 CS 40700, 38058\\
Grenoble cedex 09, France.\\
Email: narasimha-chary.bonala@univ-grenoble-aples.fr
}

\begin{document}
\maketitle
\begin{abstract} A Bott tower of height $r$  
is a sequence of projective bundles 
$$X_r  \overset{{\pi_r}}\longrightarrow X_{r-1} \overset{\pi_{r-1}}\longrightarrow \cdots \overset{\pi_2}\longrightarrow  X_1=\mathbb P^1 \overset{\pi_1} \longrightarrow X_0=\{pt\}, $$ 
where $X_i=\mathbb P (\mathcal O_{X_{i-1}}\oplus \mathcal L_{i-1})$ for a line bundle $\mathcal L_{i-1}$ over $X_{i-1}$ for all $1\leq i\leq r$ and $\mathbb P(-)$ denotes the projectivization. 
These are smooth projective toric varieties and we refer to the top object $X_{r}$ also as a Bott tower.
In this article,  we study the Mori cone and numerically effective (nef) cone
of Bott towers, and  we classify Fano, weak Fano and log Fano Bott towers. We prove some vanishing theorems for the cohomology of tangent bundle of Bott towers. 
 
\end{abstract}
\let\thefootnote\relax\footnotetext{The author is supported by AGIR Pole MSTIC project run by the University of Grenoble Alpes, France.} 

\section{Introduction}\label{intro}

In \cite{bott1958}, R. Bott and H. Samelson introduced a family of (smooth differentiable) manifolds which may be viewed as the 
total spaces of iterated $\mathbb P^1$-bundles over a point $\{pt\}$, where each $\mathbb P^1$-bundle is the projectivization of a rank $2$ decomposable vector bundle. 
In \cite{grossberg1994bott}, M. Grossberg and Y. Karshon proved (in complex geometry setting) that these manifolds have a natural action of a compact torus and also obtained some applications to representation theory and symplectic geometry.
In \cite{civan2005bott}, Y. Civan proved that these are smooth projective toric varieties.
These are called Bott towers,
we denote them by $\{(X_i, \pi_i):1\leq i \leq r\}$, where
$$X_r\overset{\pi_r}\longrightarrow X_{r-1}\overset{\pi_{r-1}}\longrightarrow \cdots \overset{\pi_2}\longrightarrow  X_1=\mathbb P^1 \overset{\pi_1}\longrightarrow \{pt\},$$
 $X_i=\mathbb P (\mathcal O_{X_{i-1}}\oplus \mathcal L_{i-1})$ for a line bundle $\mathcal L_{i-1}$ over $X_{i-1}$ for all $1\leq i\leq r$ and  $r$ is the dimension of $X_r$. 
In \cite{choi2011properties}, \cite{choi2010topological} and \cite{ishida2012filtered}, the authors studied \textquotedblleft cohomological rigidity\textquotedblright properties of Bott towers.
These also play an important role in algebraic topology and K-theory (see \cite{civan2005homotopy}, \cite{davis1991convex} and references therein).
In this article we refer to $X_r$ also as a Bott tower (it is also called Bott manifold).

In this paper we study the geometry of Bott towers in more detail by methods of toric geometry.
We work over the field $\mathbb C$ of complex numbers.
We study the {\bf Mori cone} of $X_{r}$ and 
prove that the class of curves corresponding to `primitive relations $r(P_i)$' 
forms a basis of the real vector space of numerical classes of 
one-cycles in $X_{r}$
(see Theorem \ref{Moricone} and Corollary \ref{basis}).  An extremal ray $R$ in the Mori cone is called {\it \bf Mori ray}
if $R\cdot K_{X_r}<0$, where $K_{X_r}$ is the canonical divisor in $X_r$.
We describe extremal rays and Mori rays of the Mori cone of $X_{r}$ (see Theorem \ref{mori}).
We characterize the ampleness and numerically effectiveness of  line bundles 
on $X_{r}$ (see Lemma \ref{amplenef}) and describe the generators of the {\it nef} cone of $X_{r}$ 
(see Theorem \ref{nefgenerators}).

Recall that a smooth projective variety $X$ is called {\it \bf Fano} (respectively, {\it \bf weak Fano}) if its 
 anti-canonical divisor $- K_X$ is 
 ample (respectively, {\it nef} and {\it big}).
Following \cite{anderson2014schubert}, we say  that a pair $(X, D)$ of a normal projective variety $X$ and an 
 effective $\mathbb Q$-divisor $D$ is    {\it \bf log Fano} if it is Kawamata log terminal and $-(K_X+D)$ is ample
(see Section \ref{Logfanovariety} for more details).
We study the Fano, weak Fano and the log Fano (of the pair $(X_{r}, D)$ for a suitably chosen divisor $D$ in 
 $X_{r}$) properties of the Bott tower $X_{r}$.
To describe these results we need some notation.
It is known that a Bott tower $\{(X_i, \pi_i) :1\leq i \leq r\}$ is uniquely determined by an upper 
triangular 
matrix $M_r$ with integer entries, defined via the first Chern class of the line bundle $\mathcal L_{i-1}$ on $X_{i-1}$, 
where $X_i=\mathbb P(\mathcal O_{X_{i-1}}\oplus \mathcal L_{i-1})$ for $1\leq i \leq r$ (see \cite[Section 2.3]{grossberg1994bott}, \cite{civan2005bott} and \cite[Section 7.8]{buchstaber2015toric}).
For more details see Section \ref{preliminaries}.
Let $$M_r:=\begin{bmatrix}
        1 & \beta_{12} & \beta_{13} & \dots & \beta_{1r}\\
        0 & 1 & \beta_{23} & \dots & \beta_{2r}\\
        0 & 0 & 1 &\dots & \beta_{3r}\\
        \vdots & \vdots & &\ddots & \vdots \\
        0 & \dots & \dots & & 1
       \end{bmatrix}_{r\times r},
$$ where $\beta_{ij}$'s are integers.  
 Define for $1\leq i \leq r$,  
 \[\eta^+_i:=\{r\geq j> i: \beta_{ij}>0\}\] and \[\eta^-_i:=\{r\geq j> i: \beta_{ij}<0\}.\]
  If $|\eta_{i}^-|=1$ (respectively, $|\eta_{i}^-|=2$),
 then set $\eta^-_i=\{l\}$ (respectively,  $\eta^-_i=\{l_1, l_2\}$).
  The following can be viewed as a condition on $i^{th}$ row of the matrix $M_r$:
 \noindent 
 \begin{itemize}
 \item $N^{1}_i$  is the condition that  
 
 (i) $|\eta^+_i|=0$, $|\eta^-_i|\leq 1$, and if $|\eta^-_i|=1$ 
 then $\beta_{il}=-1$; or 
 
 (ii) If $|\eta^+_i|>0$, then there exists $m>i$ such that $\beta_{im}=1$, $\beta_{ij}=0$ for all $j<m$, and $\beta_{ij}=\beta_{mj}$ for all $j>m$.
 
 \item $N^{2}_i$  is the condition that 
    
    (i) Assume that $|\eta^+_i|=0$. Then $|\eta_i^-|\leq 2$, and 
  if $|\eta^-_i|=1 (\mbox{respectively,}~|\eta^-_i|=2)$ then  $\beta_{il}=-1$ or $-2$ 
  (respectively, $\beta_{il_1}=-1=\beta_{il_2}$); or
      
      (ii) If $|\eta^+_i|>0$, then one of the following holds: Let $A_{m,j}:=\beta_{im}\beta_{mj}-\beta_{ij}$.
    
    \noindent (a) There exists $m>i$ such that $\beta_{im}=1$ or $2$; $\beta_{ij}=0$ for all $j<m$ and $A_{m,j}=0$ for $j>m$.\\ 
(b) There exists $m_2>m_1>i$ such that $-\beta_{im_1}=\beta_{im_2}=1$, $\beta_{ij}=0$ for $m_1\neq j<m_2$  and $A_{m_2,j}=0$ for $j>m_2$.\\
(c) There exists $m_2>m_1>i$ such that  $\beta_{im_1}=A_{m_1, m_2}=1$, $\beta_{ij}=0$ for $j<m_1$ and $A_{m_1,j}=0$ for $m_1<j\neq m_2$.\\
(d) There exists $m_2>m_1>i$ such that $\beta_{im_1}=1=-A_{m_1, m_2}$, $\beta_{ij}=0$ for $j<m_1$,  $A_{m_1,j}=0$ for $m_1<j<m_2$ and $\beta_{m_2j}+A_{m_1,j}=0$ for $j>m_2$.

\end{itemize}

   \begin{definition}   
 We say  $X_{r}$ satisfies \underline{condition $I$} (respectively,  \underline{condition $II$})
 if
 $N^1_{i}$ (respectively, $N^2_i$)
 holds for all $1\leq i\leq r$.
   \end{definition}
  
   Note that $N_i^1\Longrightarrow N_i^2$ for all $1\leq i\leq r$.
If $X_{r}$ satisfies condition $I$, then it also satisfies conditions $II$. 
  We prove, 
  
  \begin{theorem*}[see Theorem \ref{fano}]\
  
  \begin{enumerate}
   \item 
   $X_{r}$ is Fano if and only if it satisfies 
   $I$.
\item    $X_{r}$ is weak  Fano if and only if it satisfies 
$II$.
\end{enumerate}
\end{theorem*}

As a consequence we get some vanishing results for the cohomology of tangent bundle of Bott towers
and hence local rigidity results.
Let $T_{X_{r}}$ denote the tangent bundle of $X_{r}$.
 \begin{corollary*}[see Corollary \ref{vanishing} and Corollary \ref{rigid}]\
   If $X_{r}$ satisfies 
$I$, then  $H^i(X_{r}, T_{X_{r}})=0$ for all $i\geq 1$. In particular,  
$X_r$ is locally rigid.
\end{corollary*}

 For $1\leq i \leq r$, we define some constants  $k_i$  which again depend on the given matrix $M_r$ corresponding to the Bott tower $X_{r}$ 
 (for more details see Section \ref{Logfanovariety}). 
 We prove,
 
 \begin{theorem*}[see Theorem  \ref{logfano}]\
   The pair $(X_{r}, D)$ is log Fano if and only if $k_i<0$ for all $1\leq i\leq r$. 
 \end{theorem*}
 \begin{remark}
  By using the results of this article, in \cite{Charytoric} we give some applications to Bott-Samelson-Demazure-Hansen (BSDH) variety, which can be described
also as a iterated projective line bundle,
by degeneration of this variety to a Bott tower.
Precisely, we study Fano, weak Fano, log Fano properties for BSDH varieties (see also \cite{Charyfano}). 
We obtain some vanishing theorems for the cohomology of tangent bundle (and line bundles) on BSDH varieties   
(see also \cite{Chary1}, \cite{Charynonreduced} and \cite{Chary11}).  
We also recover the results in \cite{parameswaran2016toric}.
 \end{remark}

 The paper is organized as follows:  
 In Section \ref{preliminaries}, we discuss preliminaries on Bott towers and toric varieties.
   In Section \ref{picard}, we discuss the Picard group of the Bott tower and compute the 
 relative tangent bundle.
Section \ref{Mori0} contains detailed study of primitive collections and primitive relations of the Bott tower and we
also describe the Mori cone.  
In Section \ref{ample and nef line bundles} we describe ample and {\it nef} line bundles on
the Bott tower, 
and we find the generators of the {\it nef} cone.
In Section \ref{fanoweakfano} and \ref{Logfanovariety}, we study  Fano, weak Fano and log Fan properties for Bott towers.
We 
also see some vanishing results.
In Section \ref{extremalrays}, we describe extremal rays and Mori rays for the Bott tower.

\section{Preliminaries}\label{preliminaries} In this section we 
recall toric varieties (see \cite{cox2011toric}) and Bott towers (see \cite{civan2005bott} and \cite{buchstaber2015toric}).
We work throughout the article over the field $\mathbb C$ of complex numbers. We expect that the proofs work for 
algebraically closed fields of arbitrary characteristic, but did not find appropriate references in that generality. 

\subsection{Toric varieties}

We briefly recall the structure of toric varieties from \cite{cox2011toric} (see also \cite{fulton} and \cite{oda}).
\begin{definition}
 A normal variety $X$ is called a \it{toric variety} (of dimension $n$)
 if it contains an $n$-dimensional
 torus $T$ (i.e. $T=(\mathbb C^{*})^n$) as a Zariski open subset such that the action of 
 the torus on itself by multiplication extends to an action of the torus on $X$.
\end{definition}

Toric varieties are completely described by the combinatorics of the corresponding fans. 
 We briefly recall here, 
 let $N$ be the lattice of one-parameter subgroups of $T$ and let $M$ be 
 the lattice of characters of $T$. Let $M_{\mathbb R}:=M\otimes \mathbb R$ and 
 $N_{\mathbb R}:=N\otimes \mathbb R$. Then we have a natural bilinear pairing 
  $$\langle-,-\rangle:M_{\mathbb R}\times N_{\mathbb R}\to \mathbb R.$$
  A fan $\Sigma$ in $N_{\mathbb R}$ is a collection of convex polyhedral
cones that is closed under intersections and cone faces. Let $\check \sigma$ be the dual cone of $\sigma\in \Sigma$
in $M_{\mathbb R}$.
For $\sigma\in \Sigma$, the semigroup algebra $\mathbb C[\check \sigma\cap M]$ is a normal domain and finitely generated $\mathbb C$-algebra. Then the scheme $Spec (\mathbb C[\check \sigma\cap M])$ is 
called the affine toric variety corresponding to $\sigma$.
For a given fan $\Sigma$, we can define a toric variety $X_{\Sigma}$
by gluing the affine toric varieties $Spec (\mathbb C[\check \sigma\cap M])$
as $\sigma$ varies in $\Sigma$. For all $1\leq s \leq n$, 
$$\Sigma(s):=\{\sigma\in \Sigma: dim(\sigma)=s\}.$$
For each $\rho\in \Sigma(1)$, we denote $u_{\rho}$, the generator of $\rho \cap N.$ 
For $\sigma\in \Sigma$, 
$$\sigma(1):=\Sigma (1) \cap \sigma.$$
There is a bijective correspondence between the cones in $\Sigma$ and the $T$-orbits in $X_{\Sigma}$.
For each $\sigma\in \Sigma$, the dimension $dim(O(\sigma))$ of the $T$-orbit $O(\sigma)$ corresponding to $\sigma$ is $n-dim(\sigma)$.
Let $\tau, \sigma \in \Sigma$, 
then  $\tau$ is a face of $\sigma$ if and only if $O(\sigma)\subset \overline {O(\tau)}$, 
where $\overline{O(\sigma)}$ is the closure of $T$-orbit $O(\sigma)$.
We denote $V(\sigma)=\overline {O(\sigma)}$ and it is a toric variety with the corresponding fan being Star($\sigma$), the star of
$\sigma$ which is the set of cones in $\Sigma$ which have
$\sigma$ as a face.
Let $D_{\rho}=\overline {O(\rho)}$ be the torus-invariant prime divisor in $X_{\Sigma}$ corresponding to $\rho\in \Sigma(1)$.
The group $TDiv(X_{\Sigma})$ of  $T$-invariant divisors in $X_{\Sigma}$ is given by 
\[TDiv(X_{\Sigma})=\bigoplus_{\rho\in\Sigma(1)}\mathbb ZD_{\rho} .\] 
For each $m\in M$, the character $\chi^m$ of $T$ is a rational function on $X_{\Sigma}$ and the corresponding divisor is given by 
\[div(\chi^m)=\sum_{\rho\in \Sigma(1)}\langle m, u_{\rho} \rangle D_{\rho}.\]

\subsection{Bott towers} In this section we recall some basic definitions and results on Bott towers.
Let $\mathcal L_0$ be a trivial line bundle over a single point $X_0:=\{pt\}$, and let $X_1:=\mathbb P(\mathcal O_{X_0}\oplus \mathcal L_0)$, where 
$\mathbb P(-)$ denotes the projectivization.  
Let $\mathcal L_1$ be a line bundle on $X_{1}$, then define $X_2:=\mathbb P(\mathcal O_{X_1}\oplus \mathcal L_{1})$, which is a $\mathbb P^1$-bundle over $X_1$.
Repeat this process 
$r$-times, so that each $X_i$ is a $\mathbb P^1$-bundle over $X_{i-1}$ for $1\leq i \leq r$.
We get the following:
\[
\begin{tikzcd}
 X_r=\mathbb P(\mathcal O_{X_{r-1}}\oplus \mathcal L_{r-1}) \arrow[shorten >=5pt]{d}{\pi_r}  \\
 X_{r-1}=\mathbb P(\mathcal O_{X_{r-2}}\oplus \mathcal L_{r-2}) \arrow[shorten >=5pt]{d}{\pi_{r-1}} &  \\
 \vdots  
\arrow[shorten >=5pt]{d}{\pi_2}\\
  X_{1}=\mathbb P(\mathcal O_{X_0}\oplus \mathcal L_0)\arrow[shorten >=5pt]{d}{\pi_1}\\
X_0=\{pt\}
\end{tikzcd}
\]

For each $1\leq i \leq r$, $X_i$ is a smooth projective toric variety (see \cite[Theorem 22]{civan2005bott}). 
Consider the points $[1: 0]$ and $[0: 1]$ in $\mathbb P^1$, we call them the south pole and the north pole respectively. The zero section of 
$\mathcal L_{i-1}$ gives a section $s^0_{i}: X_{i-1}\longrightarrow X_{i}$
, the south pole section; similarly, the north pole section $s^1_i: X_{i-1}\longrightarrow X_{i}$ by letting the first coordinate in $\mathbb P(\mathcal O_{X_{i-1}}\oplus \mathcal L_{i-1})$ to vanish.

Let $1\leq i \leq r$.
Since $\pi_i: X_i\longrightarrow X_{i-1}$ is a projective bundle, by a standard result on the cohomology ring of projective bundles we have the following (see \cite[Page 429]{hartshorne} for instance, and also \cite[Proposition 10.1]{milne2016etale}):     
\begin{theorem} \label{007}
  The cohomology ring $H^*(X_i, \mathbb Z)$ of $X_i$ is a free module over $H^*(X_{i-1}, \mathbb Z)$ on generators $1$ and $u_{i}$, 
 which have degree $0$ and $2$ respectively, 
 that is 
 $$H^*(X_i, \mathbb Z)=H^*(X_{i-1}, \mathbb Z){1}\oplus H^*(X_{i-1}, \mathbb Z){u_i}.$$
The ring structure is determined by the single relation
$$u_i^2=c_1(\mathcal L_{i-1})u_i,$$ 
where $c_1(-)$ denotes the first Chern class and 
the restriction of $u_i$ to the fiber $\mathbb P^1\subset X_i$ is the first Chern class of the canonical line bundle over $\mathbb P^1$.
Hence we have $$H^*(X_{i}, \mathbb Z)=H^*(X_{i-1}, \mathbb Z)[u_i]/J_i,$$
where $J_i$ is the ideal generated by $u_i^2-c_1(\mathcal L_{i-1})u_i$.
\end{theorem}

Consider the exponential sequence (see \cite[Page 446]{hartshorne}):
$$0\longrightarrow \mathbb Z \longrightarrow \mathcal O_{X_{i-1}}\longrightarrow \mathcal O_{X_{i-1}}^*\longrightarrow 0.$$
Then we get the following exact sequence:

$0 \longrightarrow H^1(X_{i-1}, \mathbb Z)\longrightarrow H^1(X_{i-1}, \mathcal O_{X_{i-1}}) \longrightarrow H^1(X_{i-1}, \mathcal O_{X_{i-1}}^*) \overset{c_1(-)}{\longrightarrow} H^2(X_{i-1}, \mathbb Z) \longrightarrow H^2(X_{i-1}, \mathcal O_{X_{i-1}}) \longrightarrow \cdots $

Since $X_{i-1}$ is toric, we have $H^j(X_{i-1}, \mathcal O_{X_{i-1}})=0$ for all $j>0$ (see \cite[Corollary 2.8]{oda}). As $H^1(X_{i-1}, \mathcal O_{X_{i-1}}^*)=Pic(X_{i-1})$, we get $c_1(-):Pic(X_{i-1})\overset{\sim}{\longrightarrow} H^2(X_{i-1}, \mathbb Z)$. 
Then we have the following:

\begin{theorem}\label{006} Each line bundle $\mathcal L_{i-1}$ on $X_{i-1}$ is determined (up to an algebraic isomorphism) by its first Chern class, which can be written as a linear combination 
$$c_1(\mathcal L_{i-1})=-\sum_{k=1}^{i-1}\beta_{ki}u_k\in H^2(X_{i-1}, \mathbb Z),$$
where $\beta_{ik}$'s are integers for $1\leq k \leq i-1$.
\end{theorem}

Then by Theorem \ref{007} and \ref{006}, by iteration, we get the following:
\begin{corollary}
  We have $$H^*(X_{r}, \mathbb Z)=\mathbb Z[u_1, \ldots, u_r]/J,$$ where $J$ is the ideal generated by 
$\{u_j^2+\sum_{i<j}\beta_{ij}u_iu_j: 1\leq j\leq r\}$ and the integers $\beta_{ij}$'s are as in Theorem \ref{006}.
\end{corollary}

Write $\{\beta_{ij}: 1\leq i<j \leq r\}$, the collection of $r(r-1)/2$ integers, as an 
upper triangular $r \times r$ matrix 
 \begin{equation}\label{A}
  M_r:=\begin{bmatrix}
        1 & \beta_{12} & \beta_{13} & \dots & \beta_{1r}\\
        0 & 1 &\beta_{23} & \dots & \beta_{2r}\\
        0 & 0 & 1 &\dots & \beta_{3r}\\
        \vdots & \vdots & &\ddots & \vdots \\
        0 & \dots & \dots & & 1
       \end{bmatrix}_{r\times r}
   \end{equation}
  Then we get the following result (see for instance
 \cite[Lemma 2.15]{grossberg1994bott} and also \cite[Section 3]{civan2005bott}).
\begin{corollary}
 There is a bijective correspondence between \{Bott towers of height $r$\}
 and \{$r\times r$ upper triangular matrices with integer entries as in (\ref{A})\}.
 \end{corollary}
  
Two Bott towers $\{(X_i, \pi_i): 1\leq i \leq r\}$ and $\{(X'_i, \pi'_i): 1\leq i \leq r\}$ are isomorphic if there exists a collection of isomorphisms $\{\phi_i: X_i\to X_i' : 1\leq i\leq r\}$ 
such that the following diagram is commutative: 
\[
\xymatrix{
X_r\ar[r]^{\pi_r}\ar[d]^{\phi_r} & X_{r-1}\ar[r]^{\pi_{r-1}}\ar[d]^{\phi_{r-1}} & \cdots \ar[r]^{\pi_2}& X_1\ar[r]^{\pi_1}\ar[d]^{\phi_1}& X_0\ar[d]^{\phi_0} \\
X'_r\ar[r]^{\pi'_r} & X'_{r-1}\ar[r]^{\pi'_{r-1}} & \cdots \ar[r]^{\pi'_2} & X'_1\ar[r]^{\pi'_1}& X'_0
}
\]
 
 \subsubsection{Toric structure on Bott tower}
Let $\{e_1^+,\ldots, e_{r}^{+}\}$ be the standard basis of the lattice $\mathbb Z^r$.
Define, for all $i\in\{1,\ldots, r\}$, 
\begin{equation}\label{e-}
 e_i^{-}:=-e_i^{+}-\sum_{j>i}\beta_{ij}e_{j}^{+},
\end{equation}
where $\beta_{ij}$'s are integers as above. Then we have the following theorem (see \cite[Section 3 and Theorem 22]{civan2005bott} and for algebraic topology setting see \cite[Theorem 7.8.7]{buchstaber2015toric}):
\begin{theorem}\label{fan}
The Bott tower $\{(X_i, \pi_i): 1\leq i \leq r\}$ corresponding to a matrix $M_r$ as in (\ref{A}) 
is isomorphic to $\{(X_{\Sigma_i}, \pi_{\Sigma_i}): 1\leq i \leq r\}$, the collection of smooth projective toric varieties corresponding to the fan $\Sigma_i$
with the $2^i$ maximal cones generated by the set of vectors 
$$\{e_j^{\epsilon}: 1\leq j \leq i ~\mbox{and}~ \epsilon \in \{+, -\} \},$$
and where $\pi_{\Sigma_i}: X_{\Sigma_i}\to X_{\Sigma_{i-1}}$ is the toric morphism induced by the projection $\overline {\pi_{\Sigma_i}}: \mathbb Z^{i}\to \mathbb Z^{i-1}$ for all $1\leq i \leq r$. 
\end{theorem}

Note that by Theorem \ref{fan}, $\Sigma_i$ has $2i$ one-dimensional cones generated by the vectors 
$$\{e_j^{+}, e_j^-: 1\leq j \leq i\},$$ and by (\ref{e-}), we can see that the 
divisors $D_{\rho_j^+}$  corresponding to $e_j^+$ for $1\leq j\leq i$ form a basis of the Picard group of $X_{i}$ (see Section \ref{picard} for more details).

\section{On Picard group of a Bott tower}\label{picard}
Now we describe a basis of the Picard group $Pic(X_{r})$ of $X_{r}$.
Let  $\epsilon \in \{+,-\}$ and for $1\leq i \leq r$, let $\rho_i^{\epsilon}$ be the
 one-dimensional cone generated by $e_i^{\epsilon}$ . 
For all $1\leq i \leq r$, we define $D_{\rho_i^{\epsilon}}$ to be the toric divisor corresponding to
the one-dimensional cone $\rho_i^{\epsilon}$. 
We prove,
\begin{lemma}\label{divisors} The set
$\{D_{\rho_i^{\epsilon}}: 1\leq i \leq r ~\mbox{and}~ \epsilon \in \{+,-\}\}$
forms a basis of 
$Pic(X_{r})$.
 
\end{lemma}

\begin{proof}
 
By Theorem \ref{fan}, using the description of the one-dimensional cones we have the following decomposition of $\Sigma(1)$:
\begin{equation}\label{3.1}
 \Sigma(1)=\{\rho_i^+: 1\leq i \leq r\}\cup \{\rho_i^-: 1\leq i \leq r\}.
\end{equation}

Again by Theorem \ref{fan}, $\{D_{\rho_i^+}: 1\leq i\leq r\}$ forms a 
basis of the Picard group $Pic(X_{r})$ of $X_{r}$.
Since 
$$0\sim div(\chi^{e_i^+})=\sum_{\rho\in \Sigma (1)}\langle u_{\rho}, e_{i}^+ \rangle D_{\rho} ,$$ by (\ref{e-}) we can see that $\{D_{\rho_i^-}: 1\leq i\leq r\}$ also forms a 
basis of $Pic(X_{r}).$ 
In general, let $\sigma \in \Sigma$ be the  maximal cone generated by 
$\{e_i^{\epsilon}: 1\leq i \leq r\}$.
Take the torus-fixed point $x^{\epsilon}$ in $X_{r}$ corresponding to the maximal cone
$\sigma$.
Let $U$ be the torus-invariant open affine neighbourhood of $x^{\epsilon}$ in $X_{r}$.
Then $U$ is an affine space of dimension $r$; in particular, $Pic(U)=0$.
Therefore, we get
$$X_{r}\setminus U=\cup_{i=1}^{r} D_{\rho_i^{\epsilon}}$$
and $Pic(X_{r})$ is generated by $\{D_{\rho_i^{\epsilon}}: 1\leq i \leq r \}$     
(see \cite[Chapter II, Proposition 3.1, page 66]{hartshorne2006ample}).
Since $\{D_{\rho_i^{\epsilon}}: 1\leq i \leq r \}$ is linearly independent
and the rank of $Pic(X_{r})$ is $r$, 
this set $\{D_{\rho_i^{\epsilon}}: 1\leq i \leq r \}$
forms a  basis of 
$Pic(X_{r})$.
\end{proof}

 By Lemma \ref{divisors}, the set $\{D_{\rho_i^+}: 1\leq i\leq r\}$ forms a basis of 
$Pic(X_{r})$.
Now we express for each $1\leq i\leq r$, $D_{\rho_i^-}$ in terms of $D_{\rho_j^+}$'s ($1\leq j\leq r$).
Let $1\leq i \leq r$, define   $h_i^{i-1}:= -\beta_{i(i-1)}$ and  
$$h_i^j:=\begin{cases}
         0 & ~\mbox{for}~ j>i.\\
         1& ~\mbox{for}~ j=i.\\
         -\sum_{k=j}^{i-1}\beta_{i k}(h_k^j) & ~\mbox{for}~ j<i.
        \end{cases}
$$
 Then we prove,
 \begin{lemma}\label{coefficient} Let $1\leq i \leq r$.
 The coefficient of $D_{\rho_j^+}$ in $D_{\rho_i^-}$ is $h_i^j$.
 \end{lemma}
\begin{proof}
 Proof is by  induction on $i$ and by using 
\begin{equation}\label{last}
 0\sim div(\chi^{e_i^+})=\sum_{\rho\in \Sigma (1)}\langle u_{\rho}, e_{i}^+ \rangle D_{\rho}.
\end{equation}

Recall the equation (\ref{e-}), $$ e_i^{-}=-e_i^{+}-\sum_{j>i}\beta_{ij}e_{j}^{+} ~\mbox{ for all $1\leq i \leq r$}.$$
If $i=1$, by (\ref{last}), we see 
$$0\sim div(\chi^{e_1^+})=D_{\rho_1^+}-D_{\rho_1^-}.$$
Then we have \begin{equation}\label{last1}
      D_{\rho_1^-}\sim D_{\rho_1^+}.
           \end{equation}

If $i=2$, by (\ref{last}) and (\ref{e-}), we see 
$$0\sim div(\chi^{e_2^+})=D_{\rho_2^+}-D_{\rho_2^-}-\beta_{21}D_{\rho_1^-}.$$
By (\ref{last1}), we get 
$$D_{\rho_2^-}\sim D_{\rho_2^+}-\beta_{21}D_{\rho_1^+}= h_2^2D_{\rho_2^+}+h_2^1D_{\rho_1^+}.$$
By induction assume that $$D_{\rho_k^-} \sim \sum_{j=1}^{r}h_k^jD_{\rho_j^+} ~\mbox{for all}~ k<i.$$
Again by (\ref{last}) and (\ref{e-}), we see
$$0\sim div(\chi^{e_i^+})=D_{\rho_i^+}-D_{\rho_i^-}-\sum_{k<i}\beta_{ik}D_{\rho_k^-}.$$
Then $$D_{\rho_i^-}\sim D_{\rho_i^+}-\sum_{k<i}\beta_{ik}D_{\rho_k^-}.$$
Hence 
$$
D_{\rho_i^-}\sim D_{\rho_i^+}-\sum_{k<i}\beta_{ik}(\sum_{j=1}^rh_k^jD_{\rho_j^+}).$$
Since $h_k^j=0$ for $k<j$, we get
$$D_{\rho_i^-}\sim D_{\rho_i^+}-\sum_{k<i}\beta_{ik}(\sum_{j=1}^{i-1}h_k^jD_{\rho_j^+}).
$$
Then $$
D_{\rho_i^-}\sim D_{\rho_i^+}+\sum_{j=1}^{i-1}(-\sum_{k=j}^{i-1}\beta_{ik}h_k^j)D_{\rho_j^+}.
$$
Therefore, we conclude that $D_{\rho_i^-}\sim D_{\rho_i^+}+\sum_{j=1}^{i-1}h_{i}^{j}D_{\rho_j^+}$ .
This completes the proof of the lemma.
\end{proof}
Let $\epsilon \in \{+, -\}$. Define $\Sigma(1)^{\epsilon}:=\{\rho_i^{\epsilon}: 1\leq i \leq r\}.$
Then $$D=\sum_{\rho\in \Sigma(1)}a_{\rho}D_{\rho}=
\sum_{\rho\in \Sigma(1)^+}a_{\rho}D_{\rho}+
\sum_{\rho\in \Sigma(1)^-}a_{\rho}D_{\rho}.$$

For $1\leq i \leq r$, let 
$g_i:=a_{\rho_i^+}+\sum_{j=i}^{r}a_{\rho_j^-}h_j^i.$
Then we have 
\begin{corollary}\label{lastcor}
$D=\sum_{\rho\in \Sigma(1)}a_{\rho}D_{\rho} \sim  \sum_{i=1}^{r}g_iD_{\rho_i^+}.$
\end{corollary}
\begin{proof}
We have $D=\sum_{\rho\in \Sigma(1)}a_{\rho}D_{\rho}=
\sum_{i=1}^ra_{\rho_i^+}D_{\rho_i^+}+
\sum_{i=1}^ra_{\rho_i^-}D_{\rho_i^-}.$
By Lemma \ref{coefficient}, we can see that
$\sum_{i=1}^ra_{\rho_i^-}D_{\rho_i^-}\sim \sum_{i=1}^ra_{\rho_i^-}(\sum_{j=1}^ih_i^j D_{\rho_j^+}). $
Then $\sum_{i=1}^ra_{\rho_i^-}D_{\rho_i^-}\sim 
\sum_{i=1}^r(\sum_{j=i}^ra_{\rho_j^-}h_j^i) D_{\rho_j^+}.$
Hence we have  $D\sim  \sum_{i=1}^r(a_{\rho_i^+}+\sum_{j=i}^ra_{\rho_j^-}h_j^i) D_{\rho_i^+}. 
$
Thus, $D\sim \sum_{i=1}^rg_iD_{\rho_i^+}$ and this completes the proof.
\end{proof}
\begin{remark}
  By Corollary \ref{lastcor}, we see some vanishing results of the cohomology of line bundles on BSDH varieties in \cite{Charytoric}.
 \end{remark}

Let $1\leq i \leq r$. We prove  the following.
\begin{lemma}
The relative tangent bundle $T_{\pi_i}$ of $\pi_i:X_{i} \to X_{i-1}$ is given by 
$$T_{\pi_i}\simeq
\mathcal O_{X_{i}}(D_{\rho_i^+}+D_{\rho_i^-})\simeq \mathcal O_{X_{i}}(\sum_{j=1}^{i-1}\beta_{ij}D_{\rho_j^-}+2D_{\rho_i^-}).$$
 
\end{lemma}
\begin{proof} 

 By definition of Bott tower, $\pi_i$ is a $\mathbb P^1$-fibration. Then the relative canonical bundle $K_{\pi_i}$ is 
given by 
$$K_{\pi_i}=\mathcal O_{X_{i}}(K_{X_{i}})\otimes \pi_i^*(\mathcal O_{X_{i-1}}(- K_{X_{i-1}}))$$
(see \cite[Corollary 24, page 56]{kleiman1980relative}).
By \cite[Theorem 8.2.3]{cox2011toric} (see also \cite[Page 74]{fulton}), we have  
$$K_{X_{\Sigma}}=-\sum_{\rho\in \Sigma(1)}D_{\rho}.$$ Then 
$$K_{\pi_i}=\mathcal O_{X_{i}}(-\sum_{\rho\in \Sigma(1)}D_{\rho})\otimes \pi_i^*(\mathcal O_{X_{i-1}}(\sum_{\rho'\in \Sigma'(1)}D_{\rho'}))~,$$  
where $\Sigma'$ is the fan of $X_{i-1}$.
Since $X_{i-1}$ smooth, any divisor of the form $D=\sum_{\rho'\in \Sigma'(1)}a_{\rho'}D_{\rho'}$ with $a_{\rho'}\in \mathbb Z$, in $X_{i-1}$ is 
Cartier. Hence the pullback $\pi_i^*(D)$ is defined 
and given by 
$$\pi_i^*(D)=\pi^*_i(\sum_{\rho'\in \Sigma'(1)}a_{\rho'}D_{\rho'})=
\sum_{\rho\in \Sigma(1)}-\varphi_D(\overline \pi_i(u_{\rho}))D_{\rho},$$
where $\varphi_D$ is the support function corresponding to the divisor $D$ (see \cite[Theorem 4.2.12]{cox2011toric} for the correspondence between support functions and Cartier divisors).
 Since the lattice map $\overline \pi_i:\mathbb Z^i\to \mathbb Z^{i-1}$ is the projection onto the first $i-1$ factors (see page 6), by definition of $u_{\rho}$ and  $e_j^-$ (see (\ref{e-})), for $\epsilon \in \{+, -\}$ we have
 $$\overline \pi_i(u_{\rho_j^{\epsilon}})=\begin{cases}
                          u_{\rho_j^{' \epsilon}} & ~\mbox{if}~ 1\leq j\leq i-1.\\
                          0 & ~\mbox{if}~ j=i.
                         \end{cases}
$$

Hence   $$-\varphi_D(\overline \pi_i(u_{\rho_j^{\epsilon}}))=\begin{cases}
                         a_{\rho_j^{' \epsilon}} & ~\mbox{if}~ 1\leq j\leq i-1.\\
                          0 & ~\mbox{if}~ j=i.
                         \end{cases}
$$
Thus we have, $$\pi_i^*(\sum_{\rho'\in \Sigma'(1)}D_{\rho'})=\sum_{\rho\in \Sigma(1)\setminus \{\rho_i^+, \rho_i^-\}}D_{\rho}.$$
Therefore, we see that 
\begin{equation}\label{relative01}
 K_{\pi_i}=\mathcal O_{X_{i}}(-D_{\rho_i^+}-D_{\rho_i^-}).
\end{equation}

By (\ref{e-}), we note that 
\begin{equation}\label{3.8}
0\sim div(\chi^{e_i^+})=D_{\rho_i^+}-D_{\rho_i^-}-\sum_{j=1}^{i-1}\beta_{ij}D_{\rho_j^-}.
\end{equation}
Since $\check K_{\pi_i}=det~ T_{\pi_i}$, we get $\check K_{\pi_i}=T_{\pi_i}$ as $\pi_{i}$ is a $\mathbb P^1$-fibration.
Therefore, the result follows from (\ref{relative01}) and (\ref{3.8}).
  \end{proof}
\begin{remark}
 By Lemma \ref{coefficient}, the relative tangent bundle $T_{\pi_i}$ can be expressed in terms of $D_{\rho_i^+}$ ($1\leq i\leq r$).
\end{remark}

The following is well known and proved here for completeness. 

\begin{lemma}\label{normal1}
Let $X$ and $Y$ be smooth varieties. 
Let $f:X\longrightarrow Y$ be a fibration with a 
section $\sigma$ and denote by $\sigma(Y)$ its image 
in $X$.
Then the restriction of the relative tangent bundle $T_{f}$ to $\sigma(Y)$
is isomorphic to the normal bundle $\mathcal N_{\sigma(Y)/X}$ of $\sigma(Y)$ in $X$.
\end{lemma}
\begin{proof}

Consider the normal bundle short exact sequence 
\begin{equation}\label{tangentsequnce}
 0\longrightarrow T_{\sigma(Y)} \longrightarrow T_{X}|_{\sigma(Y)}\longrightarrow \mathcal N_{\sigma(Y)/X}\longrightarrow 0 ,
\end{equation}

where $T_{\sigma(Y)}$ and $T_{X}$ are the tangent bundles of
$\sigma(Y)$ and
$X$ respectively.
Also consider the following short exact sequence 
\begin{equation}\label{relative1}
 0\longrightarrow T_{f}\longrightarrow T_{X} \longrightarrow f^* T_{Y}\longrightarrow 0 ~.
\end{equation}

By restricting (\ref{relative1}) to $\sigma(Y)$, since $\sigma$ is a section of $f$, we get the following short exact sequence
\begin{equation}\label{relative101}
 0\longrightarrow T_{f}|_{\sigma(Y)}\longrightarrow T_{X}|_{\sigma(Y)} \longrightarrow  T_{\sigma(Y)}\longrightarrow 0 ~.
\end{equation}

By using (\ref{tangentsequnce}) and (\ref{relative101}), we see 
$T_{f}|_{\sigma(Y)}$ is isomorphic to $\mathcal N_{\sigma(Y)/X}$.
This completes the proof.
\end{proof}

We prove, 

\begin{lemma}\label{f_r1} Let $1\leq i \leq r$.
 The normal bundle $\mathcal N_{X_{i}/X_{i-1}}$ of 
 $X_{i-1}$ in $X_{i}$ is $\check{\mathscr L_{i-1}}$, 
 where $\mathscr L_{i-1}$ is as in the definition of Bott tower and
 $\check{ \mathscr L_{i-1}}$ is denotes the dual of $\mathscr L_{i-1}$.
\end{lemma}

\begin{proof} Fix $1\leq i\leq r$ and let $\mathscr L:=\mathscr L_{i-1}$.
Recall that $\mathbb P(\mathscr E)$ is by definition $Proj(S(\mathscr E))$, $S(\mathscr E)$
is symmetric algebra of $\mathscr E=\mathcal O_{X_{i-1}}\oplus \mathscr L$ (see \cite[Page 162]{hartshorne}). 
Let $V(\mathscr L)=Spec(S(\mathscr L))$,
the geometric vector bundle associated to the locally free sheaf
(line bundle) $\mathscr L$
(see \cite[Exercise 5.18, Page 128]{hartshorne}).
Then, $V(\mathscr L)$ is an open subvariety in $\mathbb P (\mathscr E)$ and we have the following commutative diagram

\[
\begin{tikzcd}
 V(\mathscr L) \arrow[dr, "\pi"] \arrow[hook]{rr} && \mathbb P(\mathscr E)=X_{i} \arrow[dl, swap,  "\pi_r"]\\
 &   X_{i-1} \arrow[ur,bend right=45, red, swap, "s^0_i"]\arrow[lu,bend left=50, red, "\sigma_{\pi}"]
\end{tikzcd}
\]
Also note that the section $s^0_i(X_{i-1})$ of $\pi_i$ corresponding to the projection 
$\mathscr E \to \mathcal O_{X_{i}}$ is same as the 
zero section $\sigma_\pi(X_{i-1})$ of $\pi$. 
Now consider the following short exact sequence 
\begin{equation}\label{pi}
0\longrightarrow T_{\pi}\longrightarrow T_{V(\mathscr L)} \longrightarrow \pi^*T_{X_{i-1}}\longrightarrow 0.
\end{equation}

Since the restriction $T_{\pi|_{\sigma_\pi}(X_{i-1})}$ of $T_{\pi}$ to $\sigma_{\pi}(X_{r-1})$ is $\check {\mathscr L}$, 
by Lemma \ref{normal1} and by above short exact sequence (\ref{pi}) 
we see that $\mathcal N_{\sigma_{\pi}(X_{i-1})/V(\mathscr L)} \simeq \check {\mathscr L}.$
Hence we conclude that
$\mathcal N_{X_{i-1}/X_{i}}\simeq \check{\mathscr L}$
(here we are identifying $X_{i-1}$ with the section 
corresponding to the projection 
$\mathscr E=\mathcal O_{X_{i-1}}\oplus \mathscr L \to \mathcal O_{X_{i-1}}).$
This completes the proof of the lemma.
\end{proof}

 Let $1\leq i \leq r$. We prove,
 \begin{lemma}\

\begin{enumerate}
\item The toric sections of $\pi_{i}$ are given by $D_{\rho_i^{\epsilon}}, {\epsilon}\in \{+,-\}.$
 \item The normal bundle $\mathcal N_{X_{i-1}/X_{i}}$ of
 $X_{i-1}$ in $X_{i}$ is given by
  $$\mathcal N_{X_{i-1}/X_{i}}=\check {\mathscr L_{i-1}}=\mathcal O_{X_{i}}(D_{\rho_i^+}),$$
where the line bundle $\mathscr L_{i-1}$ is as in the definition of the Bott tower $X_{i}$.
\end{enumerate}
\end{lemma}
\begin{proof} 
  Proof of (1): Recall that $\pi_i$ is a $\mathbb P^1$-fibration induced by the projection $\overline \pi_i:\mathbb Z^i \to \mathbb Z^{i-1}$.
 For each cone $\sigma\in \Sigma_F$ of dimension 1 (which is a maximal cone in $\Sigma_F$, where $\Sigma_F$ denote the fan of the fiber $\mathbb P^1$), the subvariety
$V (\sigma)$ is an invariant section of $\pi_i$, which is an invariant divisor in $X_{i}.$
Hence we get two invariant divisors $V(\rho_i^+)=D_{\rho_i^+}$
and
$V(\rho_i^-)=D_{\rho_i^-}$.

Proof of (2): By Lemma \ref{f_r1}, we have 
$\mathcal N_{X_{i-1}/X_{i}}=\check{\mathscr L_{i-1}}$ and the section $X_{i-1}$ is given by the projection $\mathscr E=\mathcal O_{X_{i-1}}\oplus \mathscr L_{i-1} \to \mathcal O_{X_{i-1}}$.
Hence (2) follows from (1).
  \end{proof}

\section{Primitive relations of the Bott tower}\label{Mori0}

\subsection{Primitive collections and primitive relations}
First recall the notion of   primitive collections and primitive relations of a fan $\Sigma$, 
which are basic tools for the
 classification of Fano toric varieties due to Batyrev (see \cite{batyrev1991classification}).
 \begin{definition}\label{def} We say
   $P \subset \Sigma (1)$ is a {\it \bf primitive collection} if $P$ is not contained in 
 $ \sigma(1)$ for some $ \sigma  \in  \Sigma$
but any proper subset is. Note that if $ \Sigma$ is simplicial, 
primitive collection means that $P$ does not generate a cone in $ \Sigma$ but every proper subset does.
 \end{definition}
 
 \begin{definition}\label{def2}
  Let $P= \{ \rho_1,  \ldots,  \rho_k \}$ be a primitive collection in a complete simplicial fan $ \Sigma$.
  Recall $u_{\rho}$ is
  the primitive vector of the ray $\rho\in \Sigma$.
  Then $\sum_{i=1}^{k}u_{ \rho_i}$ is in the relative interior of a cone $\gamma_{P}$ in $\Sigma$ with a unique expression 
    \begin{eqnarray}\label{4.1} \sum_{i=1}^{k}u_{ \rho_i}= \sum_{  \rho  \in \gamma_{P}(1)}c_{ \rho}u_{ \rho}, ~~ c_{ \rho} \in  \mathbb Q_{>0}.
~~\mbox{Hence we have }~~
    \sum_{i=1}^{k}u_{ \rho_i}-(\sum_{  \rho  \in \gamma_{P}(1)}c_{ \rho}u_{ \rho})=0 .
   \end{eqnarray}
   
Then we call (\ref{4.1}) the {\it \bf primitive relation} of $X_{\Sigma}$ corresponding to $P.$ \end{definition}

Recall that $TDiv(X_{\Sigma})$ denote the group of torus-invariant divisors in $X_{\Sigma}$
(see Page 4).
Since the fan $\Sigma$ of $X_{r}$ is full dimensional,
 we have the following short exact sequence 
\begin{equation}\label{exactpic}
 0\longrightarrow M \overset{\varphi_1}{\longrightarrow}
TDiv(X_{r})=\bigoplus_{\rho\in\Sigma(1)}\mathbb ZD_{\rho} \overset{\varphi_2}{\longrightarrow} Pic(X_{r})\to 0,
\end{equation}
where the maps are given by 
$\varphi_1:m\mapsto div(\chi^m) ~~\mbox{and}~~ \varphi_2: D\mapsto \mathcal O_{X_{r}}(D)$ (see  \cite[Theorem 4.2.1]{cox2011toric}).

Now we recall some standard notations:
Let $X$ be a smooth projective variety, we define 
$$N_1(X)_{\mathbb Z}:=\{\sum_{\mbox{finite}}a_iC_i : a_{i}\in \mathbb Z , C_i ~\mbox{irreducible curve in } ~X\}/\equiv $$
where $\equiv$ is the numerical equivalence, i.e. $Z\equiv Z'$
if and only if $D\cdot Z=D\cdot Z'$ for all divisors $D$ in $X$.
We denote by $[C]$ the class of $C$ in $N_1(X)_{\mathbb Z}$.
Let $N_1(X):=N_1(X)_{\mathbb Z}\otimes \mathbb R$. 
It is a well known fact that $N_1(X)$ is a finite dimensional real vector space
(see \cite[Proposition 4, \S 1, Chapter IV]{kleiman1966toward}).
In the case where $X$ is a (smooth projective) toric variety, $N_1(X)_{\mathbb Z}$ is dual to $Pic(X)$ via the natural pairing (see \cite[Proposition 6.3.15]{cox2011toric}). 
In our case $X=X_{r}$, there are dual exact sequences:
$$0\longrightarrow M \overset{\varphi_{1}}{\longrightarrow} \mathbb Z^{\Sigma(1)} \overset{\varphi_{2}}{\longrightarrow} Pic(X_{r}) \longrightarrow 0 $$ and  
\begin{equation}\label{123}
0\longrightarrow N_{1}(X_{r})_{\mathbb Z}\overset{\varphi^*_{2}}{\longrightarrow} \mathbb Z^{\Sigma(1)}\overset{\varphi^*_{1}}{\longrightarrow} N \longrightarrow 0,
\end{equation}
where
$$ \varphi_2^*([C])= (D_{\rho}\cdot C)_{\rho\in \Sigma(1)}, \hspace{1cm} C~\mbox{is an irreducible complete curve in }~  X_{r}
$$ and 
$$ \varphi_1^*(e_{\rho})=u_{\rho}, \hspace{1cm} e_{\rho} ~\mbox{is a standard basis vector of }~ \mathbb R^{\Sigma(1)}$$
(see \cite[Proposition 6.4.1]{cox2011toric}).
Let $P$ be a primitive collection in $\Sigma$. Note that since $X_{r}$ is smooth projective, 
$P\cap \gamma_{P}(1)=\emptyset$ and 
 \begin{equation}\label{crho}
  c_{\rho}\in \mathbb Z_{>0} ~\mbox{ for all}~~ \rho\in \gamma_P(1)
 \end{equation}
 (see \cite[Proposition 7.3.6]{cox2011toric}).
As an element in $ \mathbb Z^{\sum(1)}$, we write $r(P)=(r_{ \rho})_{ \rho \in \Sigma(1)}$, where 
\begin{equation}\label{asa}
 r_{ \rho} =
\begin{cases}
1 & \text{if } \rho \in P  \\   
-c_{\rho} & \text{if } \rho \in \gamma_{P}(1)     \\ 
0& \text{otherwise }
\end{cases}
\end{equation}

Then by (\ref{4.1}) we see that
  $$\sum_{\rho\in \sum(1)}r_{\rho}u_{\rho}=0.$$ Hence by the exact sequence
  (\ref{123}) and by (\ref{crho}), we observe that $r(P)$ gives an element in $N_1(X_{r})_{\mathbb Z}$ (see \cite[Page 305]{cox2011toric}).
 We prove,

\begin{lemma}\label{primitive}
 Let $P_i:=\{\rho_i^+,\rho_i^-\}$, $1\leq i \leq r$. Then 
  $\{P_i : 1\leq i \leq r \}$ is the set of all 
 primitive collections of the fan $\Sigma$ of $X_{r}$.
\end{lemma}
\begin{proof}
  
  By Theorem \ref{fan}, the cones in the fan $\Sigma$ of $X_{r}$ 
are generated
by subsets of 
$\{e_{1}^{+}, \ldots, e_{r}^{+}, e_{1}^-,\ldots, e_{r}^-\}$
and
containing no subset of the form 
$\{e_{i}^+, e_{i}^-\}$. 
Then by Definition \ref{def}, it is clear that $P_i=\{\rho_i^+, \rho_i^-\}$ is a primitive collection for all $i$.
Also note that again by description of the cones in $\Sigma$, any primitive collection must contain a $P_i$ for some $1\leq i \leq r$. 
  
  Fix $1\leq i \leq r$. 
  Let $Q$ be a collection of one-dimensional cones such that it properly contains $P_i$, i.e. 
  there exists $1\leq j \leq r$ and $j\neq i$ such that $\rho_j^{\epsilon}\in Q\supset P_i$, $\epsilon\in \{+, -\}$.
    Assume that $Q$ is a primitive collection.
  Then by Definition \ref{def}, 
  $\{\rho_i^+, \rho_i^-\}\subset Q$ generates a cone in $\Sigma$. 
This is a contradiction to the description of the cones in $\Sigma$.
Therefore, we conclude that  $\{P_i : 1\leq i \leq r \}$ is the set of all 
 primitive collections.
 \end{proof}
Now we define the {\bf Contractible classes} from \cite{Casagrande_contractibleclasses}:
Let $X$ be a smooth projective toric variety.
We define $NE(X)_{\mathbb Z}$ in $N_1(X)$ by 
$$NE(X)_{\mathbb Z}:=\{\sum_{finite}a_iC_i : a_i\in \mathbb Z_{\geq 0} ~\mbox{and}~ C_i ~\mbox{irreducible curve in $X$}~ \}.$$
  Let $\gamma\in NE(X)_{\mathbb Z}$ be primitive 
  (i.e. the generator of $\mathbb Z_{\geq 0}\gamma $)  and such that there exists 
  some irreducible curve in $X$ having numerical class in $\mathbb Q_{\geq 0}\gamma$.
Then \begin{definition}(see \cite[Definition 2.3]{Casagrande_contractibleclasses}) The above class $\gamma$ is called
  {\it \bf contractible} if 
  there exists a toric variety $X_{\gamma}$ and an equivariant morphism $\phi_{\gamma}:X\to X_{\gamma}$, 
  surjective with connected fibers, such that 
  for every irreducible curve $C$ in $X$, 
    $$\phi_{\gamma}(C)=\{pt\}~~\mbox{if and only if }~~[C]\in \mathbb Q_{\geq 0}\gamma.$$
 
 \end{definition}
 \begin{remark}
  
Note that any contractible class is always a class of some invariant curve 
and also a primitive relation (see \cite[Theorem 2.2]{Casagrande_contractibleclasses} and \cite[Page 74]{scaramuzza2009algorithms}).

 \end{remark}

Recall the following result from 
\cite[Proposition 3.4]{Casagrande_contractibleclasses}.
\begin{proposition}\label{Cas}
Let $P=\{\rho_1,\ldots, \rho_k\}$ be a primitive collection in $\Sigma$, with the primitive relation $r(P)$:
$$\sum_{i=1}^{k}u_{\rho_i}-\sum_{\rho\in \gamma_P(1)}c_{\rho}u_{\rho}=0.$$ Then 
$r(P)$ is contractible if and only if for every primitive collection $Q$ of $\Sigma$ such that $P\cap Q\neq \emptyset$
and $P\neq Q$, the set $(Q\setminus P)\cup 
\gamma_P(1)$
contains a primitive collection. 
\end{proposition}

\subsection{Mori cone } We use the notation as above. 
Let $X$ be a smooth projective variety.
We define 
$NE(X)$ the real convex cone  in $N_1(X)$ generated by classes of irreducible curves.
The {\it \bf Mori cone } 
 $\overline{NE}(X)$ is the closure of $NE(X)$ in $N_1(X)$ and it is a strongly convex cone of maximal dimension.
 
If $X$ is a (smooth projective) toric variety, it is known that 
$NE(X)_{\mathbb Z}$ is generated by the finitely many torus-invariant irreducible curves in $X$ and  hence 
$NE(X)_{\mathbb Z}$ is a finitely generated monoid.
Hence the cone $NE(X)=\overline{NE}(X)$ is a rational polyhedral cone and
we have
$$\overline{NE}(X)=\sum_{\tau\in \Sigma(r-1)}\mathbb R_{\geq 0}[V(\tau)] ,$$
where $r=dim(X)$ and $[V(\tau)] \in N_1(X)_{\mathbb Z}$ is the class of the toric curve $V(\tau)$. This is called the Toric Cone Theorem (see \cite[Theorem 6.3.20]{cox2011toric}).
 Let $\tau\in \Sigma(r-1)$ be a wall, that is 
$\tau=\sigma\cap \sigma'$ for some $\sigma, \sigma'\in \Sigma(r)$. 
 Let $\sigma$ (respectively, $\sigma'$) is generated by $\{u_{\rho_1}, u_{\rho_2}, \ldots , u_{\rho_r}\}$
(respectively, by $\{u_{\rho_2}, \ldots, u_{\rho_{r+1}} \}$) and
let $\tau$ be generated by $\{u_{\rho_2}, \ldots , u_{\rho_r}\}.$
 Then we get a linear relation,
 \begin{equation}\label{wallrelation}
  u_{\rho_1}+\sum_{i=2}^{r}b_iu_{\rho_i}+u_{\rho_{r+1}}=0
 \end{equation}
 The relation (\ref{wallrelation}) called {\it \bf wall relation} and we have 
 
 $$D_{\rho}\cdot V(\tau)=\begin{cases}
                         b_i & ~\mbox{if}~ \rho=\rho_i ~\mbox{and}~ i\in \{2,3,\ldots, r\}\\
                         1 & ~\mbox{if}~ \rho=\rho_i ~\mbox{and}~ i\in \{1, r+1\}\\
                         0 & ~\mbox{otherwise}~
                        \end{cases}
$$
(see \cite[Proposition 6.4.4 and eq. (6.4.6) page 303]{cox2011toric}). 
 Now we describe the Mori cone $\overline{NE}(X_{r})$ of $X_{r}$ 
in terms of the primitive relations of $X_{r}$.
\begin{theorem}\label{Moricone} 
 $\overline{NE}(X_{r})_{\mathbb Z}=\sum_{i=1}^r\mathbb Z_{\geq 0}r(P_i)$.
\end{theorem}
\begin{proof} 
We have $$\overline{NE}(X_{r})=\sum_{P\in \mathscr P}\mathbb R_{\geq 0}r(P)~~ ,$$
where $\mathscr P$ is the set of all primitive collections in $X_{r}$
(see \cite[Theorem 6.4.11]{cox2011toric}).
By Lemma \ref{primitive}, $\{P_i : 1\leq i \leq r \}$ is the set of all 
 primitive collections of $X_{r}$.
 Therefore, we get $$\overline{NE}(X_{r})=\sum_{i=1}^{r}\mathbb R_{\geq 0}r(P_i)~~.$$

 By \cite[Theorem 4.1]{Casagrande_contractibleclasses}, we have 
 $$\overline{NE}(X_{r})_{\mathbb Z}=\sum_{\gamma\in \mathscr C}\mathbb Z_{\geq 0}\gamma,$$
 where $\mathscr C$ is the set of all contractible classes in $X_{r}.$
 
 By  Proposition \ref{Cas}, we can see that the primitive relations 
 $r(P_i)$ are contractible classes for $1\leq i\leq r$. 
Since any contractible class is a primitive relation,  we get $$\mathscr C =\{r(P_i):1\leq i \leq r\}.$$
Hence we conclude that $$\overline{NE}(X_{r})_{\mathbb Z}=\sum_{i=1}^{r}\mathbb Z_{\geq 0}r(P_i).$$
This completes the proof of the theorem.
\end{proof}

We have
\begin{corollary}\label{basis}
 The set $\{r(P_i): 1\leq i \leq r\}$ forms a basis of 
 $N_1(X_{r})_{\mathbb Z}$.
\end{corollary}

\begin{proof}
By Theorem \ref{Moricone},   
$\{r(P_i): 1\leq i \leq r\}$ generates the monoid $\overline {NE}(X_{r})_{\mathbb Z}$ 
and 
the cone $\overline {NE}(X_{r})$ is of dimension $r$. So $r(P_i)$ for $1\leq i \leq r$ are linearly independent. 
Also the group $N_1(X_{r})_{\mathbb Z}$ is generated by $\overline {NE}(X_{r})_{\mathbb Z}$, hence by $r(P_i)$ for $1\leq i\leq r$.
Hence these form a basis of $N_{1}(X_{r})_{\mathbb Z}$.
\end{proof}

 Next we describe the primitive relation $r(P_i)$ explicitly by finding the cone $\gamma_{P_i}$ in (\ref{4.1}) for $1\leq i \leq r$.
We also observe that these cones depend on the given matrix corresponding to the Bott tower.
We need some notation to state the result. 
 Recall the matrix $M_r$ corresponding to the Bott tower $X_r$ is 
  $$M_r=\begin{bmatrix}
        1 & \beta_{12} & \beta_{13} & \dots & \beta_{1r}\\
        0 & 1 &\beta_{23} & \dots & \beta_{2r}\\
        0 & 0 & 1 &\dots & \beta_{3r}\\
        \vdots & \vdots & &\ddots & \vdots \\
        0 & \dots & \dots & & 1
       \end{bmatrix}_{r\times r}$$
(see Section \ref{preliminaries}).  
Fix $1\leq i \leq r$.  Define:
  \begin{enumerate}

\item Let 
$r\geq j> j_1=i\geq 1$ and define 
 $a_{1,j}:=\beta_{j_1j}$.
 \item Let $r\geq j_2>j_1$ be the least integer such that $a_{1,j}>0$, then define for $j>j_2$
 $$a_{2,j}:=\beta_{ij_2}\beta_{j_2j}-\beta_{ij}.$$
 \item Let $k>2$ and let $r\geq j_k>j_{k-1}$ be the least integer such that $a_{k-1, j}<0$, then inductively, define for $j>j_k$
  $$a_{k, j}:=-a_{k-1, j_k}\beta_{j_{k}j}+a_{k-1, j}.$$
 
 \item For  $j\leq i$, $b_j:=0$, and for $j>i$ define 
 \begin{equation}\label{bj}
    b_j:=a_{l, j} ~\mbox{if}~ j_{l+1}\geq j >j_{l},  l\geq 1.
   \end{equation}

 Note that we have  $$b_{j}=\begin{cases}
                            0 & ~\mbox{for}~ j\leq i\\
                          <0 & ~\mbox{for}~ j\in\{j_3,\ldots, j_m\}\\
                          \geq 0 & ~\mbox{otherwise .}~
                          \end{cases}
$$
 
 \item Let $I_i:=\{j_1,\ldots, j_m\}$. 
 \end{enumerate}

 \begin{example}\label{example1}
Let $$M_7=\begin{bmatrix}
        1 & -1 & -1  & -1 & 2 & -1 & 2\\
        0 & 1 & 0 & 2 & -1 & 2 & -1 \\
        0 & 0 & 1 & 0 & -1 & -1 & -1 \\
        0 & 0 & 0 & 1 & -1 & 2 & -1 \\
        0 & 0 & 0 & 0 & 1 & -1 & 2 \\
        0 & 0 & 0 & 0 & 0 & 1 & -1 \\
        0 & 0 & 0 & 0 & 0 & 0 & 1 \\
        \end{bmatrix}_{7\times 7} $$
 
  Let $i=1$, then $j_1=1$ and 
(1) $a_{1, 2}=\beta_{12}=-1$ ; (2) $a_{1, 3}=\beta_{13}=-1$ ;
(3) $a_{1, 4}=\beta_{14}=-1$ ; 
(4)  $a_{1, 5}=\beta_{15}= 2$ ; (5) $a_{1, 6}=\beta_{16}=-1$ ; (6) $a_{1, 7}=\beta_{17}=2$.

Then $j_2=5$ and 
(1)   $a_{2, 6}=\beta_{15}\beta_{56}-\beta_{16}=-1 $ ; (2) $a_{2, 7}=\beta_{15}\beta_{57}-\beta_{17}= 2$ . 

 Then $j_3=6$ and 
$a_{3, 7}=-a_{2, 6}\beta_{67}+a_{2, 7}=
 -(-1)(-1)+(2)=1$.

 Therefore, $I_1=\{1, 5, 6\}$ .

 \end{example}

 Let $1\leq i \leq r$.
 Let $\mathscr A_i:=\{e_j^{\epsilon_j} : 1\leq j \leq r, b_j\neq 0 ~\mbox{and}~ $
 
 $$  \epsilon_j=\begin{cases}
             + &~\mbox{for}~ j\notin I_i\\
            - & ~\mbox{for}~ j\in I_i
            \end{cases}
  \hspace{1.5cm}\}. $$         
                  
 \begin{remark}
  Note that as $b_j=0$ for $j\leq i$, we can take $i<j\leq r$ in the definition of $\mathscr A_i$.
 \end{remark}

 Now we have,
 \begin{proposition}\label{gamma} Let $1\leq i \leq r$.   
 The cone $\gamma_{P_i}$ in the primitive relation of $X_{r}$ corresponding to $P_i$ is generated by $\mathscr A_i$.
 
 \end{proposition}

 Before going to the proof we see an example.
 
 \begin{example}\label{example}
 We use same setting as in Example \ref{example1}.
 By Lemma \ref{primitive}, we have $P_i=\{\rho_i^+, \rho_i^-\}$ for all $1\leq i \leq 7$.
  By definition of $e_i^-$(see (\ref{e-})), we have
 
\noindent (i) $e_1^-+e_1^+=e_2^++e_3^++e_4^+-2e_5^++e_6^+-2e_7^+$; (ii) $e_2^-+e_2^+=-2e_4^++e_5^+-2e_6^++e_7^+$;
(iii) $e_3^-+e_3^+=e_5^++e_7^+$; (iv) $e_4^-+e_4^+=e_5^+-2e_6^++e_7^+$;
 (v) $e_5^-+e_5^+=e_6^+-2e_7^+$; (vi) $e_6^-+e_6^+=e_7^+$;
 (vii) $e_7^-+e_7^+=0.$
 
Now we describe the cone $\gamma_{P_1}$. 
Observe that in (i) coefficient of $e_5^+$ is negative. By (v), we can see 
$$e_1^-+e_1^+=e_2^++e_3^++e_4^++2(e_5^--e_6^++2e_7^+)+e_6^+-2e_7^+.$$
Then $e_1^-+e_1^+=e_2^++e_3^++e_4^++2e_5^--e_6^++2e_7^+.$
By (vi), 
\begin{equation}\label{rho} 
e_1^-+e_1^+=e_2^++e_3^++e_4^++2e_5^-+e_6^-+e_7^+.
\end{equation}

In this case, $I_1=\{1, 5, 6\}$ (see Example \ref{example1}) and the cone $\gamma_{P_1}$ is generated by 
$$\{e^+_2, e^+_3, e^+_4, e^-_5, e^-_6, e^+_7 \}.$$
 
 \end{example}

 Now we prove Proposition \ref{gamma}:
 \begin{proof}
By (\ref{e-}), for all $1\leq i \leq r$, we have 
\begin{equation}\label{4.2}
 e_i^-+e^+_i=-\sum_{j>i}\beta_{ij}e^+_{j} ~.
 \end{equation}
 If for all $j>i$, $\beta_{ij}\leq 0$, then the cone $\gamma_{P_i}$ is generated by 
  $\{e^+_{j} :j>i, \beta_{ij}<0\}.$
If not, choose the least integer $j_2>i$ such that $\beta_{ij_2}>0$. 
 Now write 
  $$e_i^-+e^+_i=-(\sum_{j_2>j>i}\beta_{ij}e^+_{j})+\beta_{ij_2}(-e^+_{j_2})-(\sum_{j>j_2}\beta_{ij}e^+_{j}).$$
 Again by using (\ref{4.2}), we have 
 $$e_i^-+e^+_i=
-(\sum_{j_2>j>i}\beta_{ij}e^+_{j})+\beta_{ij_2}(e_{j_2}^-+\sum_{j>j_2}\beta_{j_2j}e_{j}^+)
-(\sum_{j>j_2}\beta_{ij}e^+_{j}).
$$
Then $$e_i^-+e^+_i=
-(\sum_{j_2>j>i}\beta_{ij}e^+_{j})+\beta_{ij_2}e_{j_2}^-+\sum_{j>j_2}(\beta_{ij_2}\beta_{j_2j}-\beta_{ij})e_{j}^+).
$$

By definition $a_{2, j}=\beta_{ij_2}\beta_{j_2j}-\beta_{ij}$, then we have 
\begin{equation}\label{new1} 
 e_i^-+e^+_i=
-(\sum_{j_2>j>i}\beta_{ij}e^+_{j})+\beta_{ij_2}e_{j_2}^-
+(\sum_{j>j_2}a_{2, j}e^+_{j}).
\end{equation}
If $a_{2, j}\geq 0$ for all $j>j_2$, then $\gamma_{P_i}$ is generated by
$$\{e^{\epsilon_j}_j: j>i, 
\epsilon_j=+ \forall j \neq j_2,~\mbox{and}~  \epsilon_j=- ~\mbox{for}~ j=j_2\}.$$

Otherwise, choose the least integer $j_3>j_2$ such that $a_{2,j_3}<0$. 
By substituting $-e^+_{j_3}$ from (\ref{4.2}), we get 
$$e_i^-+e^+_i=
-(\sum_{j_2>j>i}a_{1,j}e^+_{j})+\beta_{ij_2}e_{j_2}^-
+(\sum_{j_3>j>j_2}a_{2,j}e^+_{j})-a_{2,j_3}(e^-_{j_3}+\sum_{j>j_3}\beta_{j_3j}e^+_j)+(\sum_{j>j_3}a_{2,j}e^+_{j}).$$

Then,
$$e_i^-+e^+_i=
-(\sum_{j_2>j>i}a_{1,j}e^+_{j})+\beta_{ij_2}e_{j_2}^-
+(\sum_{j_3>j>j_2}a_{2,j}e^+_{j})-a_{2,j_3}e^-_{j_3}+\sum_{j>j_3}(-a_{2, j_3}\beta_{j_3j}+a_{2, j})e^+_j).$$

By definition  $a_{3, j}=-a_{2, j_3}\beta_{j_3j}+a_{2,j}$, then we have 
\begin{equation}\label{new2}
e_i^-+e^+_i=
-(\sum_{j_2>j>i}a_{1,j}e^+_{j})+\beta_{ij_2}e_{j_2}^-
+(\sum_{j_3>j>j_2}a_{2,j}e^+_{j})-a_{2,j_3}e^-_{j_3}+(\sum_{j>j_3}a_{3,j}e^+_{j}).\end{equation}
By repeating this process, we get the cone $\gamma_{P_i}$ as we required.
\end{proof}
Let $1\leq i \leq r$. Recall $I_i=\{i=j_1, \ldots, j_m\}$ as in page 13.
  Define  for $1\leq j\leq r$ ,
 \[c_{j}:=
 \begin{cases}
-b_j & ~\mbox{if}~ j\in I_i\setminus \{j_1,j_2\}\\
b_j & ~\mbox{otherwise}~
 \end{cases}
\]

Set $\gamma_{P_i}(1):=\{\gamma_1, \ldots, \gamma_l\}$.
Then we have 
\begin{corollary}\label{curves1} For $1\leq i \leq r$, 
 the primitive relation $r(P_i)(=(r_{\rho})_{\rho \in \Sigma(1)})$ of $X_{r}$ given by  
  $$r_{\rho}=\begin{cases}
           1 & ~\mbox{for}~ \rho=\rho_i^+ ~\mbox{or}~ \rho_{i}^-\\
 -c_{j}& ~\mbox{for}~ \rho=\gamma_j\in \gamma_{P_i}(1)\\
 0& ~\mbox{otherwise}~
          \end{cases}
$$
 
\end{corollary}

\begin{example} We use Example \ref{example}. 
The following can be seen easily  from (\ref{rho}).

 \begin{enumerate}
  \item $\gamma_{P_1}(1)=\{\rho_2^+, \rho_{3}^+, \rho_{4}^+, \rho_{5}^-,\rho_{6}^-,\rho_{7}^+\}$ .\\
 
\item
The primitive relation $r(P_1)=(r_{\rho})_{\rho\in \Sigma(1)}$ is given by 
  $$r_{\rho}=\begin{cases}
           1 & ~\mbox{for}~ \rho=\rho_1^+ ~\mbox{or}~ \rho_{1}^-\\
 -1 & ~\mbox{for}~ \rho= \rho_k^+, k\in \{2, 3, 4, 7\} ~\mbox{and}~ \rho=\rho_6^-\\
 -2 & ~\mbox{for}~ \rho=\rho_5^-\\
 0 & ~\mbox{otherwise}
          \end{cases}
$$
 
 \end{enumerate}

\end{example}

Now we describe the primitive relations $r(P_i)$ in terms of  intersection of two maximal cones in the fan of $X_{r}$.
Let $1\leq i\leq r$.
Let $\mathscr C'_i:= \{e_j^{\epsilon_j}: 1\leq j\leq r ~\mbox{and}~$

$$\epsilon_j=\begin{cases}
           + & ~\mbox{if}~  j\notin I_i\setminus \{j_1\} \\ 
         - & ~\mbox{if}~ j\in I_i 
         \end{cases}
\hspace{1.5cm} \}. $$

 Let $\mathscr C''_i:= \{e_j^{\epsilon_j}: 1\leq j\leq r ~\mbox{and}~$ 
$$\epsilon_j=\begin{cases}
           + & ~\mbox{if}~  j\notin I_i \\
         - & ~\mbox{if}~ j\in I_i 
         \end{cases}
\hspace{1.5cm} \}. $$

\begin{example}
 We use Example \ref{example}, for $i=1$, we have $I_1=\{1, 5, 6\}$. Then 
 $$\mathscr C'_1=\{e_1^+, e_2^+, e_3^+, e_4^+, e_5^-, e_6^-, e_7^+\} ~\mbox{and}~ 
\mathscr C''_1=\{e_1^-, e_2^+, e_3^+, e_4^+, e_5^-, e_6^-, e_7^+\}.$$

\end{example}

We prove the following by using {\it wall relation} (see page 12).
\begin{proposition}\label{LI1}
Fix $1\leq i \leq r$. The class of curve $r(P_i)$ is given by 
  $$r(P_i)=[V(\tau_i)] ~,$$ where $\tau_i=\sigma\cap \sigma'$ and $\sigma$ (respectively,  $\sigma'$)
   is the cone generated by   $\mathscr C'_i$ (respectively, by $\mathscr C''_i$).

\end{proposition}

\begin{proof} 
 From Corollary \ref{curves1}, we have the following.
 \begin{equation}\label{4.3}
  e_i^++e_i^--\sum_{j>i}c_je_j^{\epsilon_j}=0,
 \end{equation}
 where $\epsilon_j$ is as in Proposition \ref{gamma}. 
   First we show that the set $Q:=\{\rho\in \Sigma(1): D_{\rho}\cdot V(\tau_i)>0\}$ is not contained in $\sigma(1)$
  for any $\sigma\in \Sigma$ (we adapt the arguments of \cite[Proof of Theorem 6.4.11, page 306]{cox2011toric}, here we are not assuming the curve $V(\tau_i)$ 
  is extremal).
    Indeed, suppose $Q\subseteq \sigma(1)$ for some $\sigma \in \Sigma$.
  Let $D$ be an ample divisor in $X_{r}$ (such exists as $X_{r}$ is projective). Then, we can assume that 
  $D$ is of the form 
  $$D=\sum_{\rho\in\Sigma(1)}a_{\rho}D_{\rho}~,~ a_{\rho}=0 ~\mbox{for all}~\rho\in \sigma(1)~\mbox{and}~ a_{\rho}\geq 0~\mbox{for all}~ \rho\notin \sigma(1)$$
  (see \cite[(6.4.10), page 306]{cox2011toric}).
Then we can see $$D\cdot V(\tau_i)=\sum_{\rho\notin \sigma(1)}a_{\rho}D_{\rho}\cdot V(\tau_i).$$
As $Q\subseteq \sigma(1)$, by definition of $Q$, $D_{\rho}\cdot V(\tau_i)\leq 0$ for $\rho\notin \sigma(1)$.
Since $a_{\rho}\geq 0$ for $\rho\notin \sigma(1)$, we get $D\cdot V(\tau_i)\leq 0$, which is a contradiction as $D$ is ample.
Therefore, $Q$ is not contained in $\sigma(1)$ for any $\sigma\in \Sigma$.  
   Hence to prove the proposition it is enough to prove
  $$P_i=Q (:=\{\rho\in \Sigma(1): D_{\rho}\cdot V(\tau_i)>0\})$$
 (see again \cite[Proof of Theorem 6.4.11, page 306]{cox2011toric}).
  From (\ref{4.3}) and by using {\it wall relation},
 we can see that 
 $$D_{\rho}\cdot V(\tau_i)=\begin{cases}
                          1 & ~\mbox{if}~ \rho=\rho_i^+ ~\mbox{or}~ \rho_i^-.\\
                          -c_j & ~\mbox{if}~ \rho=\rho_j^{\epsilon_j} ~\mbox{and}~ j\in I_i\setminus\{j_1\}.\\
                          0 & ~\mbox{otherwise.}~
                        \end{cases}
$$ 

Since $c_j$'s are all positive integers (see (\ref{crho})), by Lemma \ref{primitive} we conclude that  
$$P_i=\{\rho\in \Sigma(1): D_{\rho}\cdot V(\tau_i)>0\}$$ and hence $r(P_i)=[V(\tau_i)]$.
This completes the proof of the proposition. 
\end{proof}

\begin{example}
 In Example \ref{example}, the curve $r(P_1)=[V(\tau_1)]$ with
$\tau_1=\sigma\cap \sigma'$ where $\sigma$ is the cone generated by 
$$\mathscr C'_1=\{e_1^+, e_2^+, e_3^+, e_4^+, e_5^-, e_6^-, e_7^+\}$$ and 
$\sigma'$ is the cone generated by 
$$\mathscr C''_1=\{e_1^-, e_2^+, e_3^+, e_4^+, e_5^-, e_6^-, e_7^+\}.$$
\end{example}
 
 \begin{corollary}
   $\overline{NE}(X_{r})_{\mathbb Z}=\sum_{i=1}^r\mathbb Z_{\geq 0}[V(\tau_i)]$, where $\tau_i$ is as in Proposition \ref{LI1}.
 \end{corollary}
\begin{proof}
 This follows from Theorem \ref{Moricone} and Proposition \ref{LI1}
\end{proof}

\section{Ample and {\it nef} line bundles on the Bott tower}\label{ample and nef line bundles}

Let $X$ be a smooth projective variety.
Recall $N^1(X)$ is the real finite dimensional vector space of numerical classes of 
real divisors in $X$ (see \cite[ \S 1, Chapter IV]{kleiman1966toward}).
In $N^1(X)$, we define the {\it nef} cone $Nef(X)$ to be the cone generated by classes of numerically effective divisors and
it is a strongly convex closed cone in $N^1(X)$.
The ample cone $Amp(X)$ of $X$ is the cone in $N^1(X)$ 
generated by classes of ample divisors. 
Note that the ample cone $Amp(X)$ is interior of the 
{\it nef} cone $Nef(X)$ (see \cite[Theorem 1, \S 2, Chapter IV]{kleiman1966toward}).
Recall that the {\it nef} cone $Nef(X)$ and the Mori cone
 $\overline{NE}(X)$ 
 are closed convex cones and are dual to each other
 (see \cite[\S 2, Chapter IV]{kleiman1966toward} )
 .

In our case, we have $Pic(X_{r})_{\mathbb R}=N^1(X_{r})$, as the numerical equivalence and 
linear equivalence coincide (see \cite[Proposition 6.3.15]{cox2011toric}).

In this section, we characterize the ampleness and numerically effectiveness of line bundles 
on $X_{r}$ and we study the generators of the {\it nef} cone of $X_{r}$.
We use the notation as in Section \ref{Mori0}.
Let $D=\sum a_{\rho}D_{\rho}$ be a toric divisor in $X_{r}$ and 
 for $1\leq i\leq r $, define
$$d_i:=(a_{\rho_i^+}+a_{\rho_i^-}-\sum_{\gamma_j \in \gamma_{P_i}(1)}c_{j}a_{\gamma_j}).$$
Then we prove,
\begin{lemma}\label{amplenef}\

 \begin{enumerate}
  \item The divisor $D$ is ample if and only if $d_i>0$ for all $1\leq i \leq r$.
  \item  The divisor $D$ is numerically effective ({\it nef}) if and only if 
  $d_i\geq 0$ for all $1\leq i \leq r$.
 \end{enumerate}

\end{lemma}
\begin{proof}
Proof of (2): Recall that the primitive relation $r(P_i)$ is given by $$r(P_i)=(r_{\rho})_{\rho\in \Sigma(1)}$$
(see page 11). First observe that we have the following 
 $$D\cdot r(P_i)=\sum_{\rho \in \Sigma(1)}a_{\rho} (D_{\rho}\cdot r(P_i))=\sum_{\rho\in \Sigma(1)}a_{\rho}r_{\rho}$$
 (see \cite[Proposition 6.4.1, page 299]{cox2011toric}).
 Then by (\ref{asa}), we get $$D\cdot r(P_i)=\sum_{\rho\in P_i}a_{\rho}-\sum_{\rho\in \gamma_{P_i}(1)}r_{\rho}a_{\rho}.$$
 
By Lemma \ref{primitive}, we have $P_i=\{\rho_i^+, \rho_i^-\}$.
Then by Corollary \ref{curves1}, we get 

\begin{equation}\label{nnef}
D\cdot r(P_i)=(a_{\rho_i^+}+a_{\rho_i^-}-\sum_{\gamma_j\in \gamma_{P_i}(1)}c_{j}a_{\gamma_j})=:d_i.
 \end{equation}

 Since the {\it nef} cone $Nef(X_{r})$ and the Mori cone $\overline{NE}(X_{r})$ are dual to each other,
the divisor $D$ is {\it nef} if and only if $D\cdot C\geq 0$ for all torus-invariant irreducible curves $C$ in $X_{r}$.
 By Theorem \ref{Moricone}, we have $$\overline{NE}(X_{r})=\sum_{i=1}^{r}\mathbb R_{\geq 0}r(P_i).$$
  Hence $D$ is {\it nef} if and only if $D\cdot r(P_i)\geq 0$ for all $1\leq i\leq r$. 
 Therefore, by (\ref{nnef}), we conclude that the divisor $D$ is {\it nef} if and only if $d_i\geq 0$ for all $1\leq i\leq r.$ 
This completes the proof of (2).

 Proof of (1): Recall that the divisor $D$ is ample if and only if its class in $Pic(X_{r})_{\mathbb R}$ lies in the interior
 of the {\it nef} cone $Nef(X_{r})$. 
 Hence by using similar arguments as in the proof of  $(2)$ and
  the toric Kleiman criterion for ampleness \cite[Theorem 6.3.13]{cox2011toric}, we can see that
  $D$ is ample if and only if $d_i>0$ for all $1\leq i\leq r.$
\end{proof}

Next we describe the generators of the {\it nef} cone $Nef(X_{r})$ of $X_{r}$.

\begin{example}\label{ex1}
 Let $M_2=\begin{bmatrix}
 1 & -1\\
 0 & 1
\end{bmatrix}_{2\times 2}$.
 Then $X_{2}=\mathbb P(\mathcal O_{\mathbb P^1}\oplus \mathcal O_{\mathbb P^1}(1))$, 
 the Hirzebruch surface 
 $\mathscr H_1$ and the rays $\rho_1^+, \rho_1^-, \rho_2^+$ and $\rho_2^-$ of the fan (shown below) of $X_{2}$
are generated by $e_1^+, e^-_1 =-e_1^++e_2^+, e_2^+$ and $-e_2^+$ respectively.  
 
\begin{center}


\begin{tikzpicture}[auto]
\draw[thick, ->] (0,0) -- (4,0) node[anchor=north west]{$\rho_1^+$}; 
\draw[thick, ->] (0,0) -- (0,4) node[anchor= south east]{$\rho_2^+$};
\draw[thick, ->] (0,0) -- (-4,4) node[anchor= south east]{$\rho_1^-$} ;
\draw[thick, ->] (0,0) -- (0,-4) node [anchor= north west]{$\rho_2^-$}  ;
\draw (0.5,0)--(0, 0.5);
\draw (1,0)--(0, 1);
\draw (1.5,0)--(0, 1.5);
\draw (2,0)--(0, 2);
\draw (2.5,0)--(0, 2.5);
\draw (3,0)--(0, 3);
\draw (3.5,0)--(0, 3.5);
\draw (-1.1,1.1)--(0, 1.1);
\draw (-0.6,0.6)--(0, 0.6);
\draw (-1.7,1.7)--(0, 1.7);
\draw (-2.2,2.2)--(0, 2.2);
\draw (-2.7,2.7)--(0, 2.7);
\draw (-3.2,3.2)--(0, 3.2);
\draw (-3,3)--(0, -3);
\draw (-2.5,2.5)--(0, -2.5);
\draw (-3.5,3.5)--(0, -3.5);
\draw (-1,1)--(0, -1);
\draw (-2,2)--(0, -2);
\draw (-1.5,1.5)--(0, -1.5);
\draw (-0.4,0.4)--(0, -0.4);
\draw (-3.5,3.5)--(0, -3.5);
\draw (1.2,0)--(0, -1.2);
\draw (2.2,0)--(0, -2.2);
\draw (1.7,0)--(0, -1.7);
\draw (0.4,0)--(0, -0.4);
\draw (2.7,0)--(0, -2.7);
\draw (3.3, 0)--(0, -3.3);

\end{tikzpicture}

{\bf Figure. } Fan of Hirzebruck surface $\mathscr {H}_1$.

\end{center}

The primitive relations $r(P_1)$ and $r(P_2)$ are given by 
$$r(P_1)~:~ e_1^++e_1^-=e_2 ~\mbox{and}~ ~ r(P_2)~:~ e_2^++e^-_2=0 .$$
By {\it wall relation}, we observe that 

\begin{enumerate}
 \item $D_{\rho_1^+}\cdot r(P_1)=1$ and $D_{\rho_1^+}\cdot r(P_2)=0$.\\
 
 \item $D_{\rho_2^-}\cdot r(P_1)=0$ and $D_{\rho_2^-}\cdot r(P_2)=1$.
\end{enumerate}

Then the dual basis of $\{r(P_1) ~,~ r(P_2)\}$ is 
$\{D_{\rho_1^+} ~,~ D_{\rho_2^-} \}.$
Hence the generators of the {\it nef} cone $Nef(\mathscr H_1)$ are $D_{\rho_1^+}$ and $D_{\rho_2^-}$.
Note that by Lemma \ref{divisors}, $Pic(\mathscr H_1)$ is generated by $\{D_{\rho_1^+}, D_{\rho_2^-} \}.$
Let $D=aD_{\rho_1^+}+bD_{\rho_2^-}\in Pic(\mathscr H_1)$. Then
$$D~\mbox{is ample if and only if }~ a>0 ~\mbox{and}~ b>0 $$(this gives back  \cite[Example (6.1.16), page 273]{cox2011toric}).
\end{example}

Now we prove the similar results for $X_{r}$.
For $1\leq m \leq r$, define $$J_m:=\{1\leq i < m: \{\rho_m^+\}\cap \gamma_{P_i}(1) \neq \emptyset\}.$$
\begin{remark}
 Note that the set $J_m$ is the collection of indices $i<m$ 
 for which $u_{\rho_m^+}$
appear in the $\gamma_{P_i}$ part of the expression (\ref{4.1}) for the primitive relation $r(P_i).$  
\end{remark}

We set 
$D_1:=D_{\rho_1^+}$, and for $m>1$ define inductively

$$D_{m}:=\begin{cases}
         D_{\rho_m^+} & ~\mbox{if}~ J_m=\emptyset\\
         (\sum_{k\in J_m}c_{\rho_m^+}^{\gamma_{P_k}}D_k)+D_{\rho_m^+} & ~\mbox{if}~ J_m \neq \emptyset,
        \end{cases}
$$
where $-c_{\rho_m^+}^{\gamma_{P_k}}$ is the coefficient of $e_m^+$ in the primitive relation $r(P_k)$.

\begin{example}
 In Example \ref{ex1}, $D_1=D_{\rho_1^+}$, $J_2=\{1\}$ and 
 $D_2=D_1+D_{\rho_2^+}$.
  By using (\ref{e-}), we see that 
$0\sim div(\chi^{e_1^+})\sim D_{\rho_1^+}-D_{\rho_1^-}$ and 
 $0\sim div(\chi^{e_2^+})\sim D_{\rho_2^+}-D_{\rho_2^-}+D_{\rho_1^-}.$
Hence $D_2=D_1+D_{\rho_2^+}=D_{\rho_2^-}$. 
\end{example}
\begin{example}
 In Example \ref{example},  
  \begin{enumerate}
   \item Recall  by (\ref{rho}), we have  $e_1^-+e_1^+=e_2^++e_3^++e_4^++2e_5^-+e_6^-+e_7^+.$  
 Then, $\gamma_{P_1}(1)=\{\rho_2^+, \rho_{3}^+, \rho_{4}^+, \rho_{5}^-,\rho_{6}^-,\rho_{7}^+\}.$\\
 
  \item $\gamma_{P_2}(1)=\{\rho_4^-, \rho_5^-, \rho_6^+,  \rho_7^+\}$  \hspace{1cm}(since $e_2^++e_2^-=2e_4^-+e_5^-+e_6^++e^+_7$) .\\
  
  \item $\gamma_{P_3}(1)=\{\rho_5^+, \rho_7^+\}$  \hspace{2cm}(since $e_3^++e^-_3=e_5^++e_7^+ )$ .\\
  
   \item $\gamma_{P_4}(1)=\{\rho_5^+, \rho_6^-, \rho_7^-\}$ 
   \hspace{2cm}(since $e_4^++e_4^-=e^+_5+2e_6^-+e_7^-$ ). \\
  
  \item $\gamma_{P_5}(1)=\{\rho_6^+, \rho_7^-\}$ \hspace{2cm}(since $e_5^++e_5^-=e_6^++2e_7^-$ ).\\
  
  \item $\gamma_{P_6}(1)=\{\rho_7^+\}$ \hspace{2cm}(since $e_6^++e_6^-=e_7^+$ ).  \\
  
  \item $\gamma_{P_7}(1)=\emptyset.$  \hspace{2cm} (since $e_7^++e_7^-=0$ ). 
   
   \end{enumerate}
   
   Then , 
   
   \begin{enumerate}
   \item If $m=1$, then $D_1=D_{\rho_1^+}$.\\
   
   \item If $m=2$, then $J_2=\{ 1 \}$ and $c_{\rho_2^+}^{\gamma_{P_1}}=1$. Hence $D_2=D_1+D_{\rho_2^+}$.\\
     
     \item If $m=3$, then $J_3=\{ 1 \}$ and $c_{\rho_3^+}^{\gamma_{P_1}}=1$. Hence 
     $D_3=D_1+D_{\rho_3^+}$.\\
     
     \item If $m=4$, then $J_4=\{1 \}$ and $c_{\rho_4^+}^{\gamma_{P_1}}=1$. Hence $D_4=D_1+D_{\rho_{4}^+}$.\\

     \item If $m=5$, then $J_5=\{3, 4 \}$ and $c_{\rho_5^+}^{\gamma_{P_3}}=1$ ~ ; ~ $c_{\rho_5^+}^{\gamma_{P_4}}=1$. Hence 
          $$D_5=D_3+D_4+D_{\rho_{5}^+}.$$

     \item If $m=6$, then $J_6=\{2 , 5 \}$ and 
     $c_{\rho_6^+}^{\gamma_{P_2}}=1 ~ ; ~ c_{\rho_6^+}^{\gamma_{P_5}}=1$. Hence 
     $$D_6=D_2+D_5+D_{\rho_{6}^+}.$$

     \item If $m=7$, then $J_7=\{1 , 2, 3, 6 \}$ and 
     
     $$c_{\rho_7^+}^{\gamma_{P_1}}=1 ~ ; ~
     c_{\rho_7^+}^{\gamma_{P_2}}=1 ~ ; ~
     c_{\rho_7^+}^{\gamma_{P_3}}=1 ~ ; ~\mbox{and}~
     c_{\rho_7^+}^{\gamma_{P_6}}=1 ~ . ~\mbox{Hence}~$$ 
     
      $$D_7=D_1+D_2+D_3+D_6+D_{\rho_{7}^+} . $$

     \end{enumerate}

\end{example}

We prove, 

\begin{proposition} \label{dualbasis}
 The set $\{D_i:1\leq i\leq r\}$ is dual basis of $\{r(P_i):1\leq i\leq r\}$.
\end{proposition}

\begin{proof} Fix $1\leq i \leq r$.
 By Proposition \ref{LI1}, the class of curve corresponding to the primitive relation $r(P_i)$ is given by 
 $$r(P_i)=[V(\tau_i)]$$ (where $\tau_i$ is described as in  Proposition \ref{LI1}).
  From Corollary \ref{curves1}, the primitive relation $r(P_i)(=[V(\tau_i)])$ is 
 \begin{equation}\label{dual}
  e_i^++e_i^--\sum_{j>i}c_je_j^{\epsilon_j}=0,
 \end{equation} 
where $\epsilon_j$ is as in Proposition \ref{LI1}.
 Note that this is the {\it wall relation} for the torus-invariant curve $V(\tau_i)$.
 We prove   
 \begin{equation}\label{dual1}
  D_m\cdot r(P_i)=D_m\cdot V(\tau_i)=\begin{cases}
                  1 &~\mbox{if}~ i=m.\\
                  0 & ~\mbox{if}~ i\neq m.
                 \end{cases}
 \end{equation}

 By (\ref{dual}) and by {\it wall relation},
 we have  
 \begin{equation}\label{dual2} 
D_{\rho_m^+}\cdot V(\tau_i)=\begin{cases}
                  1 & ~\mbox{for}~ m=i\\
                  0 & ~\mbox{for}~ m<i\\
                  -c_{\rho_m^+}^{\gamma_{P_i}} & ~\mbox{for}~ m>i ~\mbox{and}~ i\in J_m\\                  
                  0 & ~\mbox{for}~  m>i ~\mbox{and}~ i\notin J_m
                                   \end{cases}
 \end{equation}                               
Hence by definition of $D_m$, it is clear that 
\begin{equation}\label{dual3}
 D_m\cdot V(\tau_i)=\begin{cases}
                  1 & ~\mbox{for}~ m=i\\
                  0 & ~\mbox{for}~ m<i
                   \end{cases}               
   \end{equation}

 Now we claim  $D_m\cdot V(\tau_i)=0$ for all $m>i$.
    Assume that $m>i$ and write $m=i+j$, where $1 \leq j \leq r-i.$ 
   We prove the claim by induction on $j$.
If $j=1$, then $D_m=D_{i+1}$.
 
 Case 1: If $J_{i+1}=\emptyset$, then $D_{i+1}=D_{\rho_{i+1}^+}$. By (\ref{dual2}), we see that
 $D_{i+1}\cdot V(\tau_i)=0.$
 
 Case 2: Assume that $J_{i+1}\neq \emptyset$.
 
 Subcase 1: If $i\notin J_{i+1}$, then by (\ref{dual2}) and (\ref{dual3}), we can see that $D_{i+1}\cdot V(\tau_i)=0.$
 
 Subcase 2: If $i\in J_{i+1}$, then by (\ref{dual3}), we have 
 $D_{i+1}\cdot V(\tau_i)= c_{\rho_{i+1}^+}^{\gamma_{P_i}}+ (D_{\rho_{i+1}^+} \cdot V(\tau_i)). $
 
 By (\ref{dual2}), $D_{\rho_{i+1}^+} \cdot V(\tau_i)=-c_{\rho_{i+1}^+}^{\gamma_{P_i}}$ and hence $D_{i+1}\cdot V(\tau_i)=0$. This proves the claim for $j=1$.
 
 Now assume that $j>1$. 
 
 Case 1:  If $J_{m}=\emptyset$, then by (\ref{dual2}) and (\ref{dual3}), we see that
$D_{m}\cdot V(\tau_i)=0.$
 
 Case 2: Assume that $J_{m}\neq \emptyset$.
 
 Subcase 1: If $i\notin J_{m}$, then by (\ref{dual2}) and (\ref{dual3}), 
 we can see that 
 $$D_{m}\cdot V(\tau_i)=((\sum_{k\in J_m, k>i}c_{\rho_m^+}^{\gamma_{P_k}}D_k)\cdot V(\tau_i))+(D_{\rho_m^+}\cdot V(\tau_i)).$$
  By induction on $j$, $D_k\cdot V(\tau_i)=0$ for all $i<k<m$.
 By (\ref{dual2}), as $m>i$ and $m\notin J_m$ , we have $D_{\rho_m^+}\cdot V(\tau_i)=0~.$
  Hence we conclude that 
 $D_{m}\cdot V(\tau_i)=0.$
 This completes the proof of the proposition.
 \end{proof}
 
 We have,
 
\begin{theorem}\label{nefgenerators}\

\begin{enumerate}

\item The {\it nef} cone $Nef(X_{r})$ of $X_{r}$ is generated by $\{D_i:1\leq i\leq r\}$.
 \item 
 The divisor $D=\sum_ia_iD_i$ is ample if and only if $a_i>0$ for all $1\leq i\leq r$.
\end{enumerate}
  \end{theorem}
\begin{proof}
 Since the {\it nef} cone $Nef(X_{r})$ is dual of the Mori cone $\overline{NE}(X_{r})$, (1) follows from Proposition \ref{dualbasis}.
 
 Proof of (2): This follows from (1) as the ample cone $Amp(X_{r})$ is interior of the {\it nef} cone $Nef(X_{r})$.
\end{proof}

 \section{Fano and weak Fano Bott towers}\label{fanoweakfano}
 In this section we describe the matrices $M_r$ such that the corresponding Bott tower $X_r$ is Fano or  weak Fano.
 First recall the {\it Iitaka dimension} of a Cartier divisor $D$ in a normal projective variety $X$. 
Let 
$$N(D):=\{m\geq 0 : H^0(X, \mathscr L(mD))\neq 0\},$$
where  $\mathscr L(mD) ~\mbox{is the line bundle associated to } mD.$
For $m\in N(D)$, we have a rational map
 $$\phi_m : X \dashrightarrow \mathbb P(H^0(X, \mathscr L(mD))^*).$$
If $N(D)$ is empty we define the {\it Iitaka dimension} $\kappa (D)$ of $D$ 
as 
 $-\infty$. Otherwise we define 
 $$\kappa (D):= \max \limits_{m\in N(D)} ~\{dim(\phi_m(X))\}.$$
Observe that $\kappa(D)\in \{-\infty, 0, 1, \ldots, dim(X)\}.$  
We say $D$ is {\it big} if $\kappa(D)=dim(X)$ (see \cite[Section 2.2, page 139]{lazarsfeld2004positivity}).
Note that an ample divisor is {\it big} .

 \begin{lemma}\label{big1}
 Let $X$ be a smooth projective variety, let $U$ be an open affine subset of $X$. Let 
 $D$ be an effective divisor with support $X\setminus U$. Then $D$ is {\it big}.
   \end{lemma}
\begin{proof}
It suffices to show that there exists an effective divisor $E$ with support $X\setminus U$ such that $E$ is {\it big}.
Indeed, we then have $mD=E+F$ for some $m\geq 0$ and for some effective divisor $F$.
Then $E+F$ is big and hence so is $D$.

There exists $f_1,\ldots, f_n\in \mathcal O_X(U)$ algebraically independent over $\mathbb C$, where 
$n=dim(X)$. View $f_1,\ldots, f_n$ as rational functions on $X$, then $f_1,\ldots, f_n\in H^0(X, \mathcal O_X(E))$
for some effective divisor $E$ with support $X\setminus U$ (since $div(f_i)$ is an effective divisor with support in $X\setminus U$ for $1\leq i \leq r$).
Thus, the monomials in $f_1,\ldots, f_n$ of any degree $m$ are linearly independent elements of $H^0(X, \mathcal O_X(mE))$. 
So $dim(H^0(X, \mathcal O_X(mE)))$ grows like $m^n$ as $m\to \infty$.
Hence $E$ is {\it big} (see \cite[Corollary 2.1.38 and Lemma 2.2.3]{lazarsfeld2004positivity}) and this completes the proof.
\end{proof}

We get the following as a variant of Lemma \ref{big1}.
\begin{corollary}\label{big}
Let $X$ be a smooth projective variety and $D$ be an effective divisor. Let $supp(D)$ denotes the support of $D$.
If $X\setminus supp(D)$ is affine, then $D$ is {\it big}. 
\end{corollary}

A smooth projective variety $X$ is called Fano (respectively, weak Fano) if its anti-canonical line bundle $-K_X$ is 
 ample (respectively, {\it nef} and {\it big}). 
  To describe our results 
we use the notation and terminology from Section \ref{intro} (see page 2).
  We prove,
  
  \begin{theorem}\label{fano}\
  
  \begin{enumerate}
   \item 
   $X_{r}$ is Fano if and only if it satisfies 
   $I$.
\item    $X_{r}$ is weak  Fano if and only if it satisfies 
$II$.
  \end{enumerate}

  \end{theorem}

 \begin{proof} Proof of (2):
  We have 
  \begin{equation}\label{canonical}
   K_{X_{r}}=-\sum_{\rho\in \Sigma(1)}D_{\rho}
    \end{equation}
(see \cite[Theorem 8.2.3]{cox2011toric} or  \cite[Page 74]{fulton}).
    The anti-canonical line bundle 
  of any projective toric variety is {\it big}, 
  since we have
  $$supp(-K_{X_{r}})=X_{r}\setminus (\mathbb C^*)^{r},$$
  $(\mathbb C^*)^{r}$ is an affine open subset of $X_{r}$,
by Corollary \ref{big}, $-K_{X_{r}}$ is {\it big}.

By using  Lemma \ref{amplenef}, we prove that $-K_{X_{r}}$ is {\it nef} if and only if $X_{r}$  satisfies 
$II$.

Let $D=-K_{X_r}$. 
By (\ref{canonical}) and by definition of $d_i$ for $D$ (see Lemma \ref{amplenef}), we have 
$$d_i=2-\sum_{\gamma_j\in \gamma_{P_i}(1)}c_j.$$
Then by Lemma \ref{amplenef}(2), $-K_{X_{r}}$ is {\it nef} if and only if 
$\sum_{\gamma_j\in \gamma_{P_i}(1)}c_j\leq 2$ for all $1\leq i \leq r$.

First assume that $-K_{X_{r}}$ is {\it nef}. Fix $1\leq i \leq r$.
By above discussion, we have 
\begin{equation}\label{6.2}
 \sum_{\gamma_j\in \gamma_{P_i}(1)}c_j\leq 2.
\end{equation}

Since $c_j$'s are positive integers (see (\ref{crho})), we get the following situation:
$$|\gamma_{P_i}(1)|=0 ~\mbox{or}~|\gamma_{P_i}(1)|=1, ~\mbox{or}~|\gamma_{P_i}(1)|=2.$$

\underline{Case 1}: If $|\gamma_{P_i}(1)|=0$, then by definition of $\gamma_{P_i}$ 
(see Definition \ref{def2}), we have 
$$r(P_i)~:~ e_i^++e_i^-=0.$$
Then $|\eta_i^+|=0=|\eta_i^-|$. Hence we see  
$X_{r}$ satisfies the condition $N_i^1$.

\underline{Case 2}: If $|\gamma_{P_i}(1)|=1$, 
then there exists a unique $r\geq m>i,$ 
such that $\gamma_m\in \gamma_{P_i}(1)$ and the primitive relation
 is either
\begin{equation}\label{k1}
 r(P_i)~:~ e_i^++e_i^-=c_me_m^+ 
\end{equation}
or 
\begin{equation}\label{k2}
  r(P_i)~:~ e_i^++e_i^-=c_me_m^- 
\end{equation}

 By (\ref{6.2}), we get $c_m=1 ~\mbox{or}~2$.

\underline{Subcase (i)}: Assume that $c_m=1$. 
If the primitive relation is (\ref{k1}), then
we can see that $|\eta_i^+|=0$ and $c_m=-\beta_{im}=1$. 
Then $\beta_{im}=-1$ and hence 
$X_{r}$ satisfies the condition $N_i^1$.

If the primitive relation is (\ref{k2}), then by comparing with (\ref{new1}), we get that $|\eta_i^+|>0$ and 
there exists $m>i$ such that $c_m=\beta_{im}=1$, $\beta_{ij}=0$ for all $j<m$ and $\beta_{im}\beta_{mj}-\beta_{ij}=0$ for $j>m$.
Hence $X_r$ satisfies the condition $N_i^2$.\\ 

\underline{Subcase (ii)}: Assume that $c_m=2$. 
If the primitive relation $r(P_i)$ is (\ref{k1}), then $|\eta_i^+|=0$ and $|\eta_i^-|=1$.
So by (\ref{e-}), we have $2=c_m=-\beta_{im}$.

If the primitive relation $r(P_i)$ is (\ref{k2}), then by comparing with $\ref{new1}$, we get 
 $|\eta_i^+|>0$ and there exists $m>i$ such that $c_m=\beta_{im}=2$, $\beta_{ij}=0$ for $j<m$  and $\beta_{im}\beta_{mj}-\beta_{ij}=0$ for $j>m$.\\
Hence $X_{r}$ satisfies the condition $N_i^2$.

\underline{Case 3}: If $|\gamma_{P_i}(1)|=2$, then  there exists  $r\geq m_2> m_1>i$ with
$\gamma_{m_1}, \gamma_{m_2} \in \gamma_{P_i}(1)$ such that the primitive relation $r(P_i)$ is 
\begin{equation}\label{k3}
 r(P_i)~:~ e_i^++e_i^-=c_{m_1}e_{m_1}^{\pm}+c_{m_2}e_{m_2}^{\pm} 
\end{equation}
\underline{Subcase (i):} If the primitive relation is $r(P_i) ~:~ e_i^++e_i^-=c_{m_1}e_{m_1}^{+}+c_{m_2}e_{m_2}^{+}$,
by (\ref{e-}) we see that
$|\eta_i^+|=0$ and $|\eta_i^-|=2$ .  By (\ref{6.2}) and (\ref{crho}) ($c_i$'s are positive integers), we get
$$c_{m_1}=-\beta_{im_1}=1 ~\mbox{,} ~\mbox{and}~  c_{m_2}=-\beta_{im_2}=1.$$ Hence 
$X_{r}$ satisfies the condition $N_i^2.$

\underline{Subcase (ii):} 
If the primitive relation is $r(P_i) ~:~ e_i^++e_i^-=c_{m_1}e_{m_1}^{+}+c_{m_2}e_{m_2}^{-}$,
by comparing with (\ref{new1}) we see that
$|\eta_i^+|>0$ and there exists $m_2>m_1>i$ such that $c_{m_1}=-\beta_{im_1}=1$, $c_{m_2}=\beta_{im_2}=1$, $\beta_{ij}=0$ for $m_1\neq j<m_2$  and $\beta_{im_2}\beta_{m_2j}-\beta_{ij}=0$ for $j>m_2$.

\underline{Subcase (iii):} 
If the primitive relation is $r(P_i) ~:~ e_i^++e_i^-=c_{m_1}e_{m_1}^{-}+c_{m_2}e_{m_2}^{+}$,
by (\ref{new1}) we see that
$|\eta_i^+|>0$ and there  exists $m_2>m_1>i$ such that  $c_{m_1}=\beta_{im_1}=1$, $c_{m_2}=\beta_{im_1}\beta_{m_1m_2}-\beta_{im_2}=1$, $\beta_{ij}=0$ for $j<m_1$ and $\beta_{im_1}\beta_{m_1j}-\beta_{ij}=0$ for $m_1<j\neq m_2$.

\underline{Subcase (iv):} 
If the primitive relation is $r(P_i) ~:~ e_i^++e_i^-=c_{m_1}e_{m_1}^{-}+c_{m_2}e_{m_2}^{-}$,
by comparing with (\ref{new2}), we see that
$|\eta_i^+|>0$ and there exists $m_2>m_1>i$ such that $c_{m_1}=\beta_{im_1}=1$, $1=c_{m_2}=-(\beta_{im_1}\beta_{m_1m_2}-\beta_{im_2})$, $\beta_{ij}=0$ for $j<m_1$,  $\beta_{im_1}\beta_{m_1j}-\beta_{ij}=0$ for $m_1<j<m_2$ and $\beta_{m_2j}+(\beta_{im_1}\beta_{m_1j}-\beta_{ij})=0$ for $j>m_2$.
Hence $X_r$ satisfies the condition $N_i^2$.

Therefore, we conclude that if $X_{r}$ is weak Fano then
$X_{r}$ satisfies 
the condition $II$.
Similarly, we can prove by using Lemma \ref{amplenef}(2), if $X_{r}$ satisfies 
$II$ then 
$X_{r}$ is weak Fano.
This completes the proof of (2).

   Proof of (1): This follows by using similar arguments as in the proof of (2) and Lemma \ref{amplenef}(1). 
   \end{proof}
   
   \begin{remark}
    In \cite{YS}, S. Yusuke classified Fano (and weak Fano) ``generalized Bott Manifolds''.  
   \end{remark}

 \subsection{Local rigidity of Bott towers}

Now we prove some vanishing results for 
 the cohomology of tangent bundle of the Bott tower $X_r$ and we get some local rigidity results. 
 Let $T_{X_r}$ denotes the tangent bundle of $X_r$. Then we have 
 \begin{corollary}\label{vanishing}\
If $X_{r}$ satisfies 
$I$, then  $H^i(X_{r}, T_{X_{r}})=0$ for all $i\geq 1$. 
 \end{corollary}
\begin{proof}
If $X_{r}$ satisfies 
$I$, then by Theorem \ref{fano}, $X_{r}$ is Fano variety.
By \cite[Proposition 4.2]{Bien1996}, since $X_{r}$ is a smooth Fano toric variety, we get 
 $H^i(X_{r}, T_{X_{r}})=0 ~\mbox{for all}~  i\geq 1.$ 
\end{proof}

It is well known that by Kodaira-Spencer theory, the
vanishing of $H^1(X, T_X)$ implies that $X$ is locally rigid, i.e.  admits no local deformations (see \cite[Proposition 6.2.10, page 272]{huybrechts}).
Then by above result we have
\begin{corollary}\label{rigid}
 The Bott tower $X_r$ is locally rigid if it satisfies 
 $I$.
\end{corollary}

\section{Log Fanoness of Bott towers}\label{Logfanovariety}

 Recall that a pair $(X, D)$ of a normal projective variety $X$ and an 
 effective $\mathbb Q$-divisor $D$ is {\it \bf Kawamata log terminal (klt)}
 if $K_X+D$ is $\mathbb Q$-Cartier, and for all proper birational
 maps $f:Y\longrightarrow X$, the pull back $f^*(K_X+D)=K_Y+D'$ satisfies
 $f_* K_Y=K_X$ and $\lfloor D' \rfloor\leq 0$, where  $\lfloor~\sum_i a_iD_i\rfloor=\sum_i\lfloor a_i\rfloor D_i$, $\lfloor x\rfloor$ is the greatest integer $\leq x$.
 The pair $(X, D)$ is called {\it \bf log Fano} if it is klt and $-(K_X+D)$ is ample.

  We recall here, a condition for
the anti-canonical line bundle to be {\it big} (see \cite{cascini2013anti}).
Let $X$ be a $\mathbb Q$- Gorenstein projective normal variety 
over $\mathbb C$.
If $X$ admits a divisor $D$ with the pair $(X, D)$ being log Fano 
then $-K_{X}$ is big
(In \cite{cascini2013anti} there is a necessary and sufficient condition that $X$ is log Fano (or
\textquotedblleft Fano type \textquotedblright) variety, see \cite[Theorem 1.1]{cascini2013anti} for more details on this ).

 If $X$ is smooth and $D$ is a normal crossing divisor, the pair $(X, D)$ is log Fano 
 if and only if $\lfloor D \rfloor=0$ and $-(K_X+D)$ is ample (see \cite[Lemma 2.30, Corollary 2.31 and Definition 2.34]{kollar2008birational}).
  In case of  toric variety $X$ see also \cite[Definition 11.4.23 and Proposition 11.4.24, page 558]{cox2011toric}.
  We use notation as in Lemma \ref{amplenef}.
 Let $D=\sum_{\rho\in \Sigma(1)}a_{\rho}D_{\rho}$ be a toric
 divisor in $X_{r}$, with $a_{\rho}'s$ in $\mathbb Q_{\geq 0}$
 and $\lfloor D \rfloor=0$. For $1\leq i\leq r$, define
  $$k_i:=d_i-2+\sum_{\gamma_j\in \gamma_{P_i}(1)}c_j.$$
  Then we prove,
 \begin{theorem}\label{logfano}
 
   The pair $(X_{r}, D)$ is log Fano if and only if $k_i<0$ for all $1\leq i\leq r$. 

 \end{theorem}
 \begin{proof}
  From the above discussion by the condition on $D$, the pair $(X_{r}, D)$ is log Fano
  if and only if $-(K_{X_{r}}+D)$ is ample.
    Note that as $-K_{X_{r}}=\sum_{\rho\in \Sigma(1)}D_{\rho}$, we get 
    $$-(K_{X_{r}}+D)=\sum_{\rho\in \Sigma(1)}(1-a_{\rho})D_{\rho}.$$
    By Lemma \ref{amplenef}, $-(K_{X_{r}}+D)$ is ample if and only if 
  \begin{equation}\label{5.1}
    ((1-a_{\rho_{i}^+})+(1-a_{\rho_i^-})-\sum_{\gamma_j\in \gamma_{P_i}(1)}c_j(1-a_{\gamma_j}))>0 ~\mbox{for all}~ 1\leq i\leq r.  
    \end{equation}
    Recall the definition of $d_i$ for $D$,
    $$d_i=a_{\rho_i^+}+a_{\rho_i^-}-\sum_{\gamma_j\in \gamma_{P_i}(1)}c_ja_{\gamma_j}.$$
    Then we have $$((1-a_{\rho_{i}^+})+(1-a_{\rho_i^-})-\sum_{\gamma_j\in \gamma_{P_i}(1)}c_j(1-a_{\gamma_j}))=-(d_i-2+\sum_{\gamma_j\in \gamma_{P_i}(1)}c_j) .$$
    
Hence in (\ref{5.1})    
  $$((1-a_{\rho_{i}^+})+(1-a_{\rho_i^-})-\sum_{\gamma_j\in \gamma_{P_i}(1)}c_j(1-a_{\gamma_j}))=-k_i ~\mbox{for all}~1\leq i\leq r $$
and  we conclude that $-(K_{X_{r}}+D)$ is ample if and only if $k_i<0$ for all $1\leq i\leq r$. This completes the proof of the theorem.
 \end{proof}

 \section{Extremal rays and Mori rays of the Bott tower}\label{extremalrays}
 In this section we study the extremal rays  and  Mori rays of Mori cone of 
 $X_{r}$.
  First we recall some definitions.
 Let $V$ be a finite dimensional vector space over $\mathbb R$ and let $K$ be a (closed) cone in $V$.
A subcone $Q$ in $K$ is called extremal if $u, v \in  K, u + v \in Q$ then
$ u, v \in Q$.  A face of $K$ is an extremal subcone. A one-dimensional face is called an extremal ray. Note that an extremal ray is contained in the boundary of $K$.

Let $X$ be a smooth projective variety. 
An extremal ray $R$ in $\overline {NE}(X)\subset N_1(X)$ is called Mori
if $R\cdot K_{X}<0$, where $K_X$ is the canonical divisor in $X$.
Recall that $\overline{NE}(X_{r})$ is a strongly convex rational polyhedral cone
of maximal dimension in $N_{1}(X_{r})$.
 We prove,
 \begin{theorem}\label{mori}\
 
 \begin{enumerate}
  \item The class of curves $r(P_i)$ for $1\leq i \leq r$ are all extremal rays in the Mori cone  $\overline{NE}(X_{r})$ of $X_{r}$.
  \item Fix $1\leq i \leq r$, 
  the class of curve $r(P_i)$ is  Mori ray if and only if 
  either $|\gamma_{P_i}(1)|=0$, or $|\gamma_{P_i}(1)|=1$ with $c_j=1$ for $\gamma_j\in \gamma_{P_i}(1)$.
 \end{enumerate}

 \end{theorem}

 \begin{proof} 
  Proof of (1):
 This follows from Theorem \ref{Moricone} and Corollary \ref{basis}.
 
 Proof of (2):
 By (1), $r(P_i)$ $1\leq i \leq r$ are all extremal rays in $\overline{NE}(X_{r})$. Hence for $1\leq i \leq r$, $r(P_i)$ is Mori if $K_{X{r}}\cdot r(P_i)<0$.
 Since $K_{X_{r}}=-\sum_{\rho\in \Sigma(1)}D_{\rho}$, we can see by Corollary \ref{curves1} and by similar arguments as in the proof of Lemma \ref{amplenef},
  \begin{equation}\label{mori108}
     K_{X_{r}}\cdot r(P_i)= -2+\sum_{\gamma_j\in \gamma_{P_i}(1)}c_j.
  \end{equation}

 Thus if $K_{X{r}}\cdot r(P_i)<0$, then 
  $$\sum_{\gamma_j\in \gamma_{P_i}(1)}c_j<2.$$ As $c_j$ are all positive integers (see (\ref{crho})), we get  either $|\gamma_{P_i}(1)|=0$, or $|\gamma_{P_i}(1)|=1$ and $c_j=1$ for $\gamma_j\in \gamma_{P_i}(1)$. 
  Similarly, by using (\ref{mori108}) we can prove the converse.  This completes the proof of the theorem. 
 \end{proof}

{\bf Acknowledgements:} 
I would like to thank Michel Brion  for valuable discussions,
many critical 
comments
and for encouragement throughout the preparation of this article.

\bibliographystyle{amsalpha} 


\end{document}